\documentclass[11pt,reqno,final]{amsart}

\usepackage[utf8]{inputenc}
\usepackage[marginpar=2cm]{geometry}
\usepackage{setspace}
\usepackage{indentfirst}
\usepackage{tgtermes}
\RequirePackage{lineno}
\usepackage{mathrsfs,mathtools,amssymb,stmaryrd,tikz-cd,enumitem}

\theoremstyle{plain}
\newtheorem{introthm}{Theorem}

\newtheorem{introQ}[introthm]{Question}

\newtheorem{theorem}{Theorem}[section]
\newtheorem{lemma}[theorem]{Lemma}
\newtheorem{proposition}[theorem]{Proposition}
\newtheorem{corollary}[theorem]{Corollary}
\newtheorem{claim}{Claim}

\theoremstyle{definition}

\newtheorem{example}[theorem]{Example}

\theoremstyle{remark}
\newtheorem{remark}[theorem]{Remark}

\usepackage[backend=bibtex,bibencoding=utf8,style=numeric,
url=false,doi=false,eprint=true,isbn=false,
hyperref=auto,backref=false]{biblatex}
\usepackage[notcite,notref,color]{showkeys}

\usepackage[pagebackref=false,hidelinks,unicode=true,bookmarks=true,
linktoc=page,pdfstartview={FitH},final=true]{hyperref}

\allowdisplaybreaks[3]
\binoppenalty=\maxdimen
\relpenalty=\maxdimen
\DeclareMathOperator{\id}{id}
\DeclareMathOperator{\pr}{pr}

\DeclareMathOperator{\Hom}{Hom}
\DeclareMathOperator{\End}{End}
\DeclareMathOperator{\Der}{Der}
\DeclareMathOperator{\At}{At}
\DeclareMathOperator{\image}{im}
\DeclareMathOperator{\rank}{rank}
\DeclareMathOperator{\CE}{CE}

\newcommand{\cA}{\mathcal{A}}
\newcommand{\cB}{\mathcal{B}}
\newcommand{\cE}{\mathcal{E}}
\newcommand{\cF}{\mathcal{F}}
\newcommand{\cL}{\mathcal{L}}
\newcommand{\cM}{\mathcal{M}}
\newcommand{\cN}{\mathcal{N}}
\newcommand{\cO}{\mathcal{O}}
\newcommand{\cQ}{\mathcal{Q}}

\newcommand{\frakg}{\mathfrak{g}}

\newcommand{\ZZ}{\mathbb{Z}}
\newcommand{\RR}{\mathbb{R}}

\newcommand{\argument}{\mathord{\color{black!25}-}}
\newcommand{\degree}[1]{\abs{#1}}
\newcommand{\into}{\hookrightarrow}
\newcommand{\xto}[1]{\xrightarrow{#1}}
\newcommand{\abs}[1]{\left|#1\right|}
\newcommand{\liederivative}[1]{\cL_{#1}}
\newcommand{\sections}[1]{\Gamma(#1)}
\newcommand{\XX}{\mathfrak{X}}
\newcommand{\tensor}{\otimes}

\newcommand{\atiyahcocycle}[2]{\At^{#1}_{#2}}
\newcommand{\atiyahclass}[1]{\alpha_{#1}}
\newcommand{\atiyahcocycleQ}{\atiyahcocycle{\nabla}{(\cM,Q)}}
\newcommand{\atiyahclassQ}{\atiyahclass{(\cM,Q)}}

\newcommand{\smooth}[1]{C^{\infty}({#1})}
\newcommand{\tangent}[1]{T{#1}}
\newcommand{\cotangent}[1]{T^{\vee}{#1}}
\newcommand{\tangentp}[2]{T_{#1}{#2}}

\title{The Atiyah class of DG manifolds of amplitude $+1$}

\thanks{The author is supported by the KIAS Individual Grant MG090801 and MG090802 at Korea Institute for Advanced Study.}
\thanks{\copyright{} 2026. This manuscript version is made available
under the CC-BY-NC-ND 4.0 license
\url{https://creativecommons.org/licenses/by-nc-nd/4.0/}.}

\author{Seokbong Seol}
\address{School of Mathematics, Korea Institute for Advanced Study}
\email{azuredream89@kias.re.kr}

\begin{document}

\maketitle

\begin{abstract}
A DG manifold of amplitude $+1$ encodes the derived intersection of a section $s$ and the zero section of a vector bundle $E$.
In this paper, we compute the Atiyah class of DG manifolds of amplitude $+1$. In particular, we show that the Atiyah class vanishes if and only if the intersection of $s$ with the zero section is a clean intersection. 
As an application, we study the Atiyah class of DG manifolds that encodes the derived intersection of two smooth manifolds.
\end{abstract}

\tableofcontents

\section{Introduction}
This paper investigates the Atiyah class of differential graded (DG) manifolds of amplitude $+1$. 
A DG manifold is a generalisation of a smooth manifold in which the algebra of smooth functions is endowed with a DG structure.
Such objects arise in mathematical physics in connection with BRST quantisation and AKSZ formalism \cite{MR1432574, MR2819233}. They also arise naturally in various fields of mathematics including Lie theory and differential geometry \cite{MR3000478, MR2581370, MR2971727}.

Formally, a DG manifold is a $\ZZ$-graded manifold $\cM$ equipped with a homological vector field $Q$, that is, a degree $+1$ derivation of $\smooth{\cM}$ satisfying $[Q,Q]=0$. 
Classical geometric structures such as regular foliations, complex manifolds, and Lie algebras can all be viewed as special cases of DG manifolds.

Of particular interest in this paper are DG manifolds of amplitude $+1$, often called \emph{quasi-smooth derived manifolds} in the $C^\infty$-context \cite{MR4735657,arXiv:2307.08179}. 
These are DG manifolds $(\cM, Q)$ where $\smooth{\cM} = \sections{\Lambda E^{\vee}}$ is the exterior algebra of the dual bundle of a vector bundle $E \to M$, and $Q = \iota_s$ is the interior product with a section $s \in \sections{E}$. 
Such data are denoted by $(E[-1], \iota_{s})$ and encode the derived intersection of the section $s$ with the zero section of $E$---see \cite{MR4735657}.

The Atiyah class is a central invariant in the study of DG manifolds.
It was originally introduced by Atiyah~\cite{MR86359} in the context of holomorphic vector bundles as the obstruction to the existence of holomorphic connections. 
Kapranov~\cite{MR1671737} showed that the Atiyah class of a Kähler manifold $X$ gives rise to an $L_{\infty}[1]$ algebra structure on the tangent complex $\Omega^{0,\bullet}_{X}(\mathcal{T}_{X})$, which in turn leads to the reformulation of Rozansky--Witten invariant~\cite{MR1481135, MR1671725}. 

The Atiyah class of DG manifolds was first introduced by Shoikhet~\cite{arXiv:math/9812009} in terms of Lie algebra cohomology and 1-jets of tangent bundles. It also appeared in the work of Lyakhovich, Mosman, and Sharapov~\cite{MR2608525} (denoted as $B_1$), and was studied systematically by Mehta, Stiénon, and Xu~\cite{MR3319134}. 
Extending the construction of Kapranov, the Atiyah class of a DG manifold $(\cM, Q)$ induces an $L_{\infty}[1]$ algebra structure on the space $\XX(\cM)$ of vector fields~\cite{MR4393962}. 
Moreover, the Atiyah class plays a key role in the Duflo--Kontsevich-type theorem for DG manifolds~\cite{MR3754617}, which recovers the classical Duflo theorem~\cite{MR0444841} and a result of Kontsevich on the Hochschild cohomology of complex manifolds~\cite{MR2062626}. Below, we recall its definition in terms of affine connections~\cite{MR3319134}.

Given a DG manifold $(\cM,Q)$, and an affine connection $\nabla$ on $\cM$, one defines the Atiyah $1$-cocycle 
\[\At^{\nabla}(X,Y) = [Q, \nabla_{X}Y] - \nabla_{[Q,X]}Y - (-1)^{\degree{X}} \nabla_{X}[Q,Y]\]
a degree $+1$ element in the complex of $(1,2)$-tensors on $\cM$ with differential $\liederivative{Q}$. The cohomology class $[\At^{\nabla}]$ is independent of $\nabla$ and defines the Atiyah class $\atiyahclass{(\cM,Q)}$---it is the obstruction to the existence of affine connections on $\cM$ compatible with $Q$. 

This class recovers classical invariants in special cases. For example (see~\cite{MR86359, MR0281224, MR3877426, MR4665716, MR3319134}):
\begin{itemize}
\item For a complex manifold $X$, the Atiyah class of $(\cM,Q)=(T^{0,1}_{X}[1], \overline{\partial})$ is identified with the classical Atiyah class of the holomorphic tangent bundle.
\item For a regular foliation $F$, the Atiyah class of $(\cM,Q)=(T_{F}[1], d_{F})$ is identified with the Atiyah–Molino class.
\item For a Lie algebra $\frakg$, the Atiyah class of $(\cM,Q)=(\frakg[1], d_{\CE})$ corresponds to the Lie bracket.
\end{itemize}

Despite the variety of known examples, most of the existing literature focuses on the Atiyah class for DG manifolds of \textit{negative amplitude}.
In contrast, the case of \textit{positive amplitude}, even for amplitude $+1$, has received comparatively little attention. 
This motivates our central question: 

\begin{introQ}\label{question}
What is the geometric meaning of the Atiyah class of DG manifolds of amplitude $+1$?
\end{introQ}

We answer this question by identifying the vanishing of the Atiyah class with a classical geometric condition on the corresponding intersection:

\begin{introthm}[Theorem~\ref{thm:Main}]\label{thm:A}
Let $E$ be a vector bundle, and let $s$ be a section of $E$. Then the Atiyah class of the DG manifold $(E[-1], \iota_{s})$ vanishes if and only if the intersection of $s$ with the zero section is clean.
\end{introthm}

Recall that a section $s$ intersects the zero section $\sigma_{0}$ cleanly if the intersection $Z = \image(s)\cap \image(\sigma_{0})$ is a smooth submanifold and, at each point $p \in Z$, we have
\[\tangentp{p}Z = \tangentp{p}{\image(s)} \cap \tangentp{p}{\image(\sigma_{0})}.\]

The proof of Theorem~\ref{thm:A} relies on a key structural feature of the Atiyah class in a slightly more general setting: locality for positive amplitude. That is, for DG manifolds of positive amplitude, the Atiyah class can be computed locally on the base manifold. More precisely, we show that the vanishing of the Atiyah class is equivalent to its vanishing on a sufficiently fine open cover. This behaviour stands in contrast to the negative amplitude case, where the Atiyah class may carry global obstructions not visible locally.

The underlying reason for this difference is that, if $(\cM,Q)$ is of positive amplitude, the homological vector field $Q$ acting on the $\smooth{M}$-algebra $\smooth{\cM}$ is $\smooth{M}$-linear. As a consequence, its Lie derivative acts trivially on the functions on the base manifold, which allows one to reduce cohomological computations to local ones using a partition of unity. 
This property fails in general for DG manifolds of negative amplitude. 
For instance, for the DG manifold $(T^{0,1}_{X}[1], \overline{\partial})$ arising from a complex manifold $X$, the Dolbeault operator $\overline{\partial}$ is not $\smooth{X}$-linear and the Atiyah class detects genuinely global geometric data.

We use this locality to compute the Atiyah class explicitly in local coordinates. On a small enough open set where $E$ is trivialised and $s$ is represented by smooth functions via the trivialisation, we explicitly compute the Atiyah class and show that it vanishes precisely when the section $s$ intersects the zero section cleanly. The result then follows by covering the base manifold with such neighbourhoods.

As an application, we consider the DG manifold $(\cM_{XY},Q)$ of amplitude $+1$ that models the derived intersection of two embedded submanifolds $X$ and $Y$ of a smooth manifold $W$, as in \cite{MR4735657}. This DG manifold consists of the following data, upon a choice of an affine connection on $W$:
\begin{itemize}
\item a base manifold consisting of short geodesics $\gamma$ in $W$ such that $\gamma(0)\in X$ and $\gamma(1)\in Y$,
\item a vector bundle whose fibre at $\gamma$ is the space of covariantly constant vector fields along $\gamma$,
\item a section that maps $\gamma$ to its derivative $\dot{\gamma}$,
\end{itemize}
and this construction is independent of the choice of connection, up to isomorphism of DG manifolds.
Our result then implies that the Atiyah class of this DG manifold encodes the clean intersection of $X$ and $Y$:
\begin{introthm}[Theorem~\ref{thm:DerivedIntersection}]
The Atiyah class of DG manifold $(\cM_{XY},Q)$ vanishes if and only if $X$ and $Y$ intersect cleanly.
\end{introthm}

We conclude by outlining two directions for further investigation. First, it would be natural to extend our analysis to DG manifolds of arbitrary finite positive amplitude, as studied in \cite{MR4735657}. Second, it remains to be seen whether the Atiyah class in this setting is invariant under weak equivalences of DG manifolds. A positive answer would provide further evidence of its robustness as a derived geometric invariant.

\subsection*{Notations and conventions}
Throughout this paper, the base field is the field of real numbers $\RR$: vector spaces, manifolds, vector bundles, functions in this paper are over $\RR$, unless otherwise stated.

For any smooth function $f:\RR^{n}\to \RR$, we write $f=f(x_{1},\ldots, x_{n})$. The notation $x^{i}$ denotes the coordinate function $x^{i}:\RR^{n}\to \RR$ defined by $x^{i}(x_{1},\ldots, x_{n})=x_{i}$.

We reserve the symbol $M$ for a manifold exclusively. By a manifold, we mean a smooth manifold without boundary with Hausdorff and second countability properties. The sheaf of algebra of smooth functions on $M$ is denoted by $\cO_{M}$. The algebra of smooth functions on $M$ is denoted by $\smooth{M}=\cO_{M}(M)$.

All gradings in this paper are $\ZZ$-gradings and the symbol $\cM$ is reserved for a finite-dimensional graded manifold.
The abbreviation `DG' stands for `differential graded'. 

Let $R$ be a graded ring. Given any element $v$ in a graded $R$-module $V=\bigoplus_{k\in \ZZ} V^{k}$, the symbol $\degree{v}=d$ means that $v$ is a homogeneous element of degree $d$, or equivalently $v\in V^{d}$.  Whenever the symbol 
$|v|$ appears, we assume that $v$ is a homogeneous element.
The suspension of the graded $R$-module $V$ is $V[1]$ whose homogeneous component of degree $k$ is $(V[1])^{k}=V^{k+1}$. 
Given any graded vector bundle $\cE$, we use the symbol $S\cE$ to denote the bundle of graded symmetric tensor powers of $\cE$.

We reserve the symbol $E$ for an ordinary vector bundle $E\to M$, the symbol $\sigma_{0}$ for its zero section, and the symbol $s$ for a section of $E$.
By abuse of notation, we often identify a section $s\in \sections{E}$ with a section $s\in \sections{E[-1]}$, and similarly identify
 a $\tangent{M}$-connection $\nabla:\sections{\tangent{M}}\times \sections{E}\to \sections{E}$ on $E$ with a $\tangent{M}$-connection $\nabla:\sections{\tangent{M}}\times \sections{E[-1]}\to \sections{E[-1]}$ on $E[-1]$.

\section{Background on DG manifolds}\label{sec:1}
In this section, we present background material on the Atiyah class of DG manifolds. 
We refer the reader to \cite{MR2709144,MR3319134} for further details.

Let $M$ be a smooth manifold, and let $\cO_{M}$ be its sheaf of smooth functions. 
A \textbf{graded manifold} $\cM=(M,\cA)$ consists of the base manifold $M$ and a sheaf of graded commutative $\cO_{M}$-algebras $\cA$ on $M$ such that 
there exists a $\ZZ$-graded vector space $V$ and, for each $p \in M$, an open neighbourhood $U\subset M$ of $p$ satisfying
\[\cA|_{U} \cong \big(\cO_{M}|_{U}\tensor S(V^{\vee})\big)^{sh},\]
where $S(V^{\vee})$ denotes the graded symmetric algebra of $V^{\vee}$, and the superscript $sh$ denotes sheafification.
The space of global sections of the sheaf $\cA$ will often be denoted by $\smooth{\cM}$. 

A graded manifold $\cM$ is said to be of \textbf{amplitude $[n,m]$}, for $n\leq m$, if the graded vector space $V$ is of the form
\[V = \bigoplus_{i=n}^{m}V_{i} ,\]
where each $V_{i}$ consists of vectors of degree $i$. If $0<n\leq m$, then we say that $\cM$ is of positive amplitude. An amplitude of the form $[n,n]$ will simply be called amplitude $n$.

A graded manifold $\cM$ is called \textbf{finite-dimensional} if both $\dim M<\infty$ and $\dim V < \infty$. All graded manifolds considered in this paper will be finite-dimensional.

Given two graded manifolds $\cM=(M,\cA)$ and $\cN=(N,\cB)$, a morphism of graded manifolds $\Phi: \cM\to \cN$ consists of a pair $(f,\phi)$, where $f:M\to N$ is a morphism of smooth manifolds and $\phi :f^{\ast}\cB \to \cA$ is a morphism of sheaves of graded commutative $\cO_{M}$-algebras. We denote the induced morphism of graded algebras of smooth functions by $\Phi^{\ast}:\smooth{\cM}\to \smooth{\cN}$.
We say $\Phi$ is an isomorphism if $f$ is a diffeomorphism and $\phi$ is an isomorphism.

\begin{remark}
In some literature, such as \cite{MR3754617, MR3319134, MR4393962}, the sheaf of $\cO_{M}$-algebras $\cA$ is often defined using formal power series on $V$ rather than polynomial functions. However, all results in this paper remain valid for both definitions.
\end{remark}

\begin{example}\label{ex:GradedMfd}
   A graded vector bundle of finite rank over a manifold $M$ is defined as $\mathbb{E}=\bigoplus_{n\leq i \leq m} E_{i}[-i]$ where each $E_{i}\to M$ is an ordinary vector bundle of finite rank, and the symbol $[-i]$ denotes the degree shift of the fibre by $-i$.
The degree $i$ component of the graded $\smooth{M}$-module $\sections{\mathbb{E}}$ is $\sections{E_{i}[-i]}$. 
Every graded vector bundle $\mathbb{E}$ determines a (split) graded manifold $\cM=(M,\cA)$: the sheaf of algebras of functions is given by $\cA(U)=\sections{U; S(\mathbb{E}^{\vee})}$, where $S(\mathbb{E}^{\vee})$ is the bundle of graded symmetric powers of $\mathbb{E}^{\vee}$. In this case, $\cM$ is of amplitude $[n,m]$.
\end{example}

Let $\cM=(M,\cA)$ be a graded manifold. A \textbf{graded vector bundle} $\pi:\cE \to \cM$ is defined as a vector bundle object in the category of graded manifolds (see~\cite[Section~2.2]{MR2709144}). For each integer $k$, graded vector bundles admit a degree-shift functor $[k]:\cE\mapsto \cE[k]$ obtained by shifting the degree of the fibres of $\cE$ by $k$. A homogeneous section of degree $k$ is a morphism of graded manifolds $s_{k}:\cM\to \cE[k]$ such that $\pi \circ [-k]\circ s_{k}=\id_{\cM}$. We denote the $\smooth{\cM}$-module of homogeneous sections of degree $k$ by $\sections{\cE}^{k}$. 
Consequently, the $\smooth{\cM}$-module of all sections of $\cE$ over $\cM$ is denoted by $\sections{\cM;\cE}=\sections{\cE}=\bigoplus_{k}\sections{\cE}^{k}$. 
Alternatively, in terms of sheaves, a graded vector bundle is a sheaf of locally free graded $\cA$-modules on $M$. Hence, the global sections form the $\smooth{\cM}$-module $\sections{\cE}$.

An important example of a graded vector bundle over $\cM$ is the \textbf{tangent bundle} $\tangent{\cM}$. 
A section of $\tangent{\cM}$ is called a vector field, and the space of vector fields $\sections{\tangent{\cM}}$, often denoted by $\XX(\cM)$, is identified with the space of graded derivations $\Der(\smooth{\cM})$. 
Equipped with the graded commutator, $\Der(\smooth{\cM})$ forms a Lie algebra; hence, so does $\sections{\tangent{\cM}}$.

Given a graded manifold $\cM$, a $\tangent{\cM}$-\textbf{connection} on a graded vector bundle $\cE$ is an $\RR$-bilinear map 
\[\nabla:\sections{\tangent{\cM}}\times \sections{\cE}\to \sections{\cE}\]
of degree $0$ satisfying:
\begin{enumerate}
\item $\nabla_{fX}s=f\nabla_X s$,
\item $\nabla_X (fs)= X(f)\cdot s+(-1)^{\degree{f}\degree{X}}f \nabla_X s$,
\end{enumerate}
for homogeneous $f\in\smooth{\cM}$, $X\in\XX(\cM)$ and $s\in \sections{\cE}$.

When $\cE=\tangent{\cM}$, the $\tangent{\cM}$-connection $\nabla$ is called an \textbf{affine connection}.
We say that an affine connection $\nabla$ is \textbf{torsion-free} if 
\[\nabla_{X}Y-(-1)^{\degree{X}\cdot \degree{Y}}\nabla_{Y}X = [ X,Y]\]
 for homogeneous $X,Y\in \XX(\cM)$. 
It is well-known that torsion-free affine connections always exist~\cite{MR3319134}.

A \textbf{DG manifold} is a graded manifold $\cM$ equipped with a homological vector field,
i.e., a vector field $Q\in\XX(\cM)$ of degree $+1$ satisfying $[Q,Q]=0$. Given two DG manifolds $(\cM,Q)$ and $(\cN,R)$, a morphism of DG manifolds $(\cM,Q)\to (\cN,R)$ is a morphism of graded manifolds $\Phi:\cM\to \cN$ such that 
$\Phi^{\ast}\circ R = Q\circ \Phi^{\ast}$. Such a morphism is an isomorphism if its underlying morphism of graded manifolds is an isomorphism.

A \textbf{DG vector bundle} $\pi:(\cE,Q_{\cE})\to (\cM,Q)$ is a vector bundle object in the category of DG manifolds (see \cite{MR3319134} for a precise definition). Given a DG manifold $(\cM,Q)$, a graded vector bundle $\cE\to \cM$ admits a DG vector bundle structure if and only if there exists a degree $+1$ operator $\widetilde{Q}_{\cE}:\sections{\cE}\to\sections{\cE}$ such that 
$(\sections{\cE},\widetilde{Q}_{\cE})$ is a DG $(\smooth{\cM},Q)$-module. 
When no confusion arises, the DG module structure $\widetilde{Q}_{\cE}$ on $\sections{\cE}$ will be denoted by the same symbol $Q_{\cE}$.

In particular, the tangent bundle $\tangent{\cM}$ naturally carries the structure of a DG vector bundle.
The defining operator on sections is $\liederivative{Q}= [Q, \argument]$, the Lie derivative along the homological vector field $Q$. 
When $\tangent{\cM}$ is viewed as a graded manifold, 
the corresponding homological vector field $Q_{\tangent{\cM}}$ is the tangent lift of the homological vector field $Q$ on $\cM$.

Consider the graded vector bundle $\cE= \cotangent{\cM}\tensor \End(\tangent{\cM}) \to \cM$. 
Define $Q_{\cE}$ as the degree +1 operator on the graded $\smooth{\cM}$-module
$\sections{\cotangent{\cM}\tensor \End(\tangent{\cM})}$:
\begin{equation*}
Q_{\cE}:\sections{\cotangent{\cM} \tensor \End(\tangent{\cM})}^{\bullet}
\to\sections{\cotangent{\cM} \tensor \End(\tangent{\cM})}^{\bullet+1}
\end{equation*}
given by the Lie derivative $\liederivative{Q}$ along $Q$:
\[ (Q_{\cE} F)(X,Y)=[Q,F(X,Y)]-(-1)^{k}F([Q,X],Y)-(-1)^{k+\degree X}F(X,[Q,Y]) \]
for any degree $k$ $(1,2)$-tensor field
$F\in\sections{\cotangent{\cM}\tensor\End(\tangent{\cM})}^{k}$
and for homogeneous vector fields $X,Y\in\XX(\cM)$. 
It is straightforward to verify that $(\cE,Q_{\cE})=(\cotangent{\cM}\tensor\End(\tangent{\cM}),\liederivative{Q})$ is a DG vector bundle over $(\cM,Q)$.

Now, given an affine connection $\nabla$ on $\cM$, consider the $(1,2)$-tensor $\atiyahcocycleQ\in\sections{\cotangent{\cM}
\tensor\End(\tangent{\cM})}$ of degree +1, defined by
\begin{equation}\label{eq:At}
\atiyahcocycleQ (X,Y) = [Q, \nabla_{X}Y] - \nabla_{[Q,X]}Y-(-1)^{\degree X} \nabla_{X}[Q,Y] 
\end{equation}
for homogeneous $X,Y\in \XX(\cM)$.

\begin{proposition}[\cite{MR3319134}]\label{prop:AtiyahProperty}
Let $(\cM,Q)$ be a DG manifold, and let $\nabla$ be an affine connection on $\cM$. Then the following hold:
\begin{enumerate}
\item  \label{item:S2} 
If $\nabla$ is torsion-free, then
$\atiyahcocycleQ\in\sections{S^{2}(\cotangent{\cM})\tensor\tangent{\cM}}$.
In other words,
\[\atiyahcocycleQ(X,Y)=(-1)^{\degree X\cdot\degree Y}\atiyahcocycleQ(Y,X) .\] 
\item The degree $1$ element $\atiyahcocycleQ \in \sections{\cotangent{\cM} \tensor \End(\tangent{\cM})}^{1}$ is a $1$-cocycle with respect to $\liederivative{Q}$.
\item The cohomology class $[\atiyahcocycleQ]$ does not depend on the choice of connection.
\end{enumerate}
\end{proposition}

The element $\atiyahcocycleQ$ is called the \textbf{Atiyah cocycle}
associated with the affine connection $\nabla$.
The cohomology class $\atiyahclassQ:=[\atiyahcocycleQ]\in H^{1}
\big(\sections{\cotangent{\cM} \tensor \End(\tangent{\cM})}^{\bullet},\liederivative{Q} \big)$
is called the \textbf{Atiyah class} of the DG manifold $(\cM,Q)$ \cite{MR3319134}.
See also \cite{arXiv:math/9812009} and \cite{MR2608525}.
The Atiyah class of a DG manifold $(\cM,Q)$ is the obstruction to the existence of an affine connection on $\cM$ that is compatible with $Q$. Moreover, this class is invariant under isomorphisms of DG manifolds.

\begin{remark}\label{rem:AtiyahProperty}
Equation~\eqref{eq:At} shows that the Atiyah cocycle can be viewed as $\atiyahcocycleQ=\liederivative{Q}(\nabla)$. However, the connection $\nabla$ is not a $(1,2)$-tensor; it is not $\smooth{\cM}$-linear in the second argument. Consequently, the Atiyah cocycle $\atiyahcocycleQ$ is not necessarily a coboundary, and the Atiyah class defines a non-trivial cohomology class.

We also remark that, by Proposition~\ref{prop:AtiyahProperty}~\eqref{item:S2}, the Atiyah class actually belongs to the subspace
\[H^{1}\big(\sections{S^{2}(\cotangent{\cM})\tensor\tangent{\cM}}^{\bullet}, \liederivative{Q} \big) \subset H^{1}
\big(\sections{\cotangent{\cM} \tensor \End(\tangent{\cM})}^{\bullet},\liederivative{Q} \big).\]
\end{remark}

\section{The Atiyah class and clean intersection}\label{sec:Main}

As shown in \cite{MR4735657}, the category of DG manifolds of positive amplitude is equivalent to the category of bundles of positively graded curved $L_{\infty}[1]$ algebras. 
In particular, any DG manifold of amplitude $+1$ is of the form $(\cM,Q)=(E[-1],\iota_{s})$, where $E\to M$ is a vector bundle and $\iota_{s}$ denotes the interior product with a section $s\in \sections{E}$.

Explicitly, the graded algebra of smooth functions is the algebra of sections of graded symmetric tensor powers on the graded dual of $E[-1]$; that is,
\[\smooth{E[-1]}=\sections{S(E[-1]^{\vee})}\cong \sections{S(E^{\vee}[1])},\] 
and the homological vector field is $Q=\iota_{s}$, the interior product with $s\in \sections{E}$. To match with the grading conventions, $s$ is often viewed as an element of $\sections{E[1]}$.
Together, they define a cochain complex of $\smooth{M}$-modules
\[\cdots \to \sections{S^{2}(E[-1]^{\vee})}\xto{\iota_{s}} \sections{S^{1}(E[-1]^{\vee})} \xto{\iota_{s}} \smooth{M} \to 0,\]
equipped with multiplication structure induced by the graded symmetric product, making it into a DG algebra.

\begin{remark}
The graded algebra of functions on $\cM=E[-1]$ can be identified with
\[\smooth{E[-1]}\cong \sections{\Lambda^{-\bullet}E^{\vee}},\]
where the degree $n$ component is $\sections{\Lambda^{-n}E^{\vee}}$. Although the exterior algebra (with wedge notation) is commonly used in this context, it often disagrees with other grading conventions and can therefore lead to confusion. For this reason, we intentionally avoid the wedge notation and instead work with graded symmetric tensors.
\end{remark}

Given a DG manifold $(E[-1],\iota_{s})$, there are two natural embeddings of $M$ into $E$: the section $s:M\to E$, and the zero section $\sigma_{0}: M\to E$. One can then consider their intersection,
\[Z:= \image (s) \cap \image (\sigma_{0}).\] 
It turns out that the Atiyah class of the DG manifold $(E[-1],\iota_{s})$ measures the failure of this intersection to be clean.

The following is our main theorem.
\begin{theorem}\label{thm:Main}
Let $M$ be a smooth manifold, and let $E$ be a vector bundle over $M$. Given a section $s\in \sections{E}$, the Atiyah class of the DG manifold $(E[-1],\iota_{s})$ vanishes if and only if the intersection of $s$ with the zero section $\sigma_{0}$ is clean.
\end{theorem}

\subsection{Clean intersection}
Let $X$ and $Y$ be submanifolds of $W$. 
We say that $X$ and $Y$ \textbf{intersect cleanly} if their intersection $Z:=X\cap Y$ is a submanifold and 
\[\tangentp{p}{Z}=\tangentp{p}{X}\cap \tangentp{p}{Y}, 
\quad \forall p\in Z.\]
Moreover, we say that the embeddings $f:X\to W$ and $g:Y\to W$ intersect cleanly if their images intersect cleanly.

\begin{remark}
Clean intersections of connected submanifolds may not have a well-defined dimension. More precisely, let $X$ and $Y$ be connected submanifolds of a connected manifold $W$. Even if $X$ and $Y$ intersect cleanly, their intersection $Z=X\cap Y$ may have multiple connected components---say, $Z=Z_{1}\sqcup Z_{2}$---with $\dim Z_{1}\neq \dim Z_{2}$.
\end{remark}

The following proposition which characterises clean intersection is well-known.
\begin{proposition}\cite[Proposition~C.3.1]{MR2304165} \label{prop:CleanGen}
  $X$ and $Y$ intersect cleanly if and only if, for each $p \in X\cap Y$, there exists a coordinate chart on a neighbourhood $U\subset W$ such that both $X\cap U$ and $Y\cap U$ are defined by linear equations in the local coordinates.
\end{proposition}

In this subsection, we focus in particular on intersections arising from DG manifolds of amplitude $+1$, namely, the intersection of a section $s:M\to E$ with the zero section $\sigma_{0}:M\to E$ of a vector bundle $E\to M$, and we prove an analogue of the above proposition in this setting.
Throughout this subsection, we assume that $E\to M$ is a vector bundle with $\dim M=n$ and $\rank E= m$, and that $s\in \sections{E}$.

Note that smooth sections of a smooth vector bundle are embeddings.
\begin{proposition}\label{prop:CleanIntersection}
The intersection of $s$ with the zero section $\sigma_{0}$ is clean if and only if, for each $p\in s^{-1}(0)$, there exist an open neighbourhood $U\subset M$ of $p$,
a local coordinate chart $\phi:U \to \RR^{n}$, 
a local frame $\{e_{1},\ldots, e_{m}\}$ of $E$ over $U$,
and an integer $0\leq r \leq \min(m,n)$
such that
\begin{equation}\label{eq:SectionLocalForm}
s\circ \phi^{-1}(x_{1},\ldots, x_{n})=x_{1} e_{1}+\cdots +x_{r} e_{r}
\end{equation}
for $(x_{1},\ldots,x_{n}) \in \phi(U)$.
\end{proposition}

To prove this proposition, we begin with some preparation. 

For each $p\in M$,
denote by 
\[\tangent{s}_{p}:\tangentp{p}{M}\to \tangentp{s(p)}{E}\]
the tangent map at $p$ induced by the section $s:M\to E$. If $p\in s^{-1}(0)$, there is a natural splitting 
\begin{equation}\label{eq:TE}
\tangentp{s(p)}E \cong \tangentp{p}M \oplus E_{p}.
\end{equation}
Composing $\tangent{s}_{p}$ with the natural projection $\pr: \tangentp{p}{M} \oplus E_{p}\to E_{p}$, we obtain a linear map 
\begin{equation}\label{eq:Dsp}
\begin{tikzcd}
Ds_{p}:\tangentp{p}{M} \arrow{r}{\tangent{s}_{p}} & \tangentp{s(p)}E \cong \tangentp{p}{M} \oplus E_{p} \arrow{r}{\pr} & E_{p} 
\end{tikzcd}  
\end{equation}
defined at each $p\in s^{-1}(0)$.

\begin{lemma}\label{lem:Rank}
The intersection of $s$ with the zero section $\sigma_{0}$ is clean if and only if $Z:=s^{-1}(0)$ is a submanifold and, for every $p\in Z$, 
\[\dim \tangentp{p}{Z}+\rank Ds_{p}=\dim \tangentp{p}M. \]
\end{lemma}
  
\begin{proof}
Let $X=\image(s)$ and $Y=\image(\sigma_{0})$. By definition, the intersection of the two sections $s$ and $\sigma_{0}$, viewed as embeddings, is the intersection $X\cap Y$. It follows directly from the construction that $X\cap Y$ is identified with the zero locus $Z=s^{-1}(0)$, via $s$. Hence, $X\cap Y$ is a submanifold if and only if $Z$ is a submanifold.

Next, we claim that
\[\tangentp{s(p)}{X} \cap \tangentp{s(p)}{Y} \cong \ker Ds_{p}\]
for each $p\in Z$. To see this, note that $\tangentp{s(p)}{X}=\image(\tangent{s}_{p})$ as a subspace of $\tangentp{s(p)}E$. Moreover, if $p\in Z$, then 
\[\tangentp{s(p)}{X} \cong \{(v,Ds_{p}(v))\in \tangentp{p}{M}\oplus E_{p}: \forall v\in \tangentp{p}{M}\}\]
via the splitting~\eqref{eq:TE}. 
Under the same splitting,
\[\tangentp{s(p)}Y \cong  \{(v,0)\in \tangentp{p}{M}\oplus E_{p} :  \forall v \in \tangentp{p}{M}\}.\]
Thus, $\tangentp{s(p)}X \cap \tangentp{s(p)}Y \cong \ker Ds_{p}$, via the identification through $\tangent{s}_{p}$. 

Moreover, when $Z$ is a submanifold, one can check that $\tangentp{p}{Z} \subset \tangentp{p}{M}$ satisfies $Ds_{p}(\tangentp{p}{Z})=0$, or equivalently, $\tangentp{p}{Z}\subset \ker Ds_{p}$. Therefore, 
\[\tangentp{p}{Z}=\ker Ds_{p} \quad \text{if and only if}\quad  \dim \tangentp{p}{Z} = \dim \ker Ds_{p}.\]

Finally, by the rank theorem,
the map $Ds_{p}:\tangentp{p}{M}\to E_{p}$, viewed as a morphism of vector spaces, satisfies
\[\dim \ker Ds_{p} + \rank Ds_{p} = \dim \tangentp{p}{M} .\]
Combining these with the definition of clean intersection completes the proof.
\end{proof}

To prove Proposition~\ref{prop:CleanIntersection}, we use the following technical lemma.
\begin{lemma}\label{lem:Calculus}
Let $M=V\subset \RR^{n}$ be an open neighbourhood of $0$, and let $E=V\times \RR^{m}$. Suppose that $s\in \sections{V;E}$ satisfies the following conditions:
\begin{enumerate}
\item $s(0)=0$, \label{item:1}  
\item $\rank (Ds_{0} )=r$, \label{item:2}
\item $s(x)=0$ for all $x \in V\cap (\{0\}^{r}\times \RR^{n-r})$. \label{item:3}
\end{enumerate}
Then there exists an open neighbourhood $U\subset V$ of $0$,
a local coordinate chart $\phi:U\to \RR^{n}$, and a local frame $\{e_{1},\ldots,e_{m}\}$ of $E$ over $U$ such that
\[s\circ \phi^{-1}(x_{1},\ldots,x_{n})=x_{1} e_{1}+\cdots +x_{r} e_{r}.\]
\end{lemma}

\begin{proof}
  Using the standard trivialisation of the bundle $E$, we identify $s$ with a map
  \[s=(s^{1},\ldots,s^{m}):V \to \RR^{m} ,\]
  where each $s^{i}$ is a smooth function on $V$.

According to~\cite[Theorem~2.9~(1)]{MR532830}, conditions~\eqref{item:1} and~\eqref{item:2} imply that 
  there exist an open neighbourhood $U\subset V$ of $0$ and a local coordinate chart 
  $\phi:U\to \RR^{n}$ such that $\phi(0)=0$ and, possibly after reordering the components of $s$,
  \begin{equation}\label{eq:TechLem1}
  s\circ \phi^{-1}(x_{1},\ldots,x_{n})=(x_{1},\ldots,x_{r},\widetilde{s}^{r+1}(x_{1},\ldots,x_{n}),\ldots,\widetilde{s}^{m}(x_{1},\ldots,x_{n}))  
  \end{equation}
  where each $\widetilde{s}^{i}$ for $i>r$ is a smooth function vanishing at the origin.
  
We may further assume that $U$ is convex. Condition~\eqref{item:3} implies, by Hadamard's lemma (see e.g. \cite[lemma~2.8]{MR4221224}), that for each $i>r$, there exist smooth functions $g^{i}_{k}$, with $k=1,\ldots,r$, such that
  \begin{equation}\label{eq:TechLem2}
    \widetilde{s}^{i}(x_{1},\ldots,x_{n})=\sum_{k=1}^{r}x_{k} \cdot g^{i}_{k}(x_{1},\ldots,x_{n}).
  \end{equation}
  Now define a local frame $\{e_{1},\ldots,e_{m}\}$ for $E$ by 
\[e_{k}=\begin{cases}
  \epsilon_{k} + \sum\limits_{j=r+1}^{m} g_{k}^{j} \cdot \epsilon_{j}, & \text{if } k\leq r\\
  \epsilon_{k}, & \text{if } k>r
\end{cases}
\] 
where $\epsilon_{k}$ denotes the standard frame of $E$ with $1$ in the $k$-th position and $0$ elsewhere. Combining this with Eqs.~\eqref{eq:TechLem1} and~\eqref{eq:TechLem2} completes the proof.
\end{proof}

Finally, we are ready to prove Proposition~\ref{prop:CleanIntersection}.
\begin{proof}[Proof of Proposition~\ref{prop:CleanIntersection}]
Assume that for each $p\in s^{-1}(0)$, the section $s$ admits the expression~\eqref{eq:SectionLocalForm} in a neighbourhood $U=U_{p}$ of $p$. It follows immediately that $\rank Ds_{p}=r$ and that $s^{-1}(0)$ is locally a submanifold of dimension $n-r$.
By Lemma~\ref{lem:Rank}, $s$ and $\sigma_{0}$ intersect cleanly.

Conversely, assume that $s$ and $\sigma_{0}$ intersect cleanly. Fix a point $p\in s^{-1}(0)$. Since $s^{-1}(0)$ is a submanifold, there exist an open neighbourhood $U_{0} \subset s^{-1}(0)$ and a local coordinate chart $\alpha_{0}:U_{0} \to \RR^{d}$ such that $\alpha_{0}(p)=0$. By Lemma~\ref{lem:Rank}, $\rank Ds_{p}=n-d$. 

By the local slice theorem~\cite[Theorem~5.8]{MR2954043}, possibly after shrinking $U_{0}$, there is an open neighbourhood $U\subset M$ of $U_{0}$ such that $\alpha_{0}$ extends to an embedding $\alpha: U\to \RR^{d}\times \RR^{n-d}$ in such a way that the following diagram commutes:
\[
\begin{tikzcd}
U  \arrow{r}{\sim} & \alpha(U) \arrow[hook]{r} & \RR^{d}\times \RR^{n-d} \\
U_{0} \arrow[hook]{u} \arrow{r}{\sim} & \alpha_{0}(U_{0}) \arrow[hook]{u} \arrow[hook]{r} &  \RR^{d}\times \{0\}^{n-d} \arrow[hook]{u} .
\end{tikzcd}
\]
Moreover, by shrinking $U_{0}$ and $U$ if necessary, we also have $(s|_{U})^{-1}(0)=U_{0}$. 

Let $V=\alpha(U) \subset \RR^{n}$. Consider the section
\[s_{V}:V\cong U\xto{s} E|_{U}\cong V\times \RR^{m}\]
of the trivial bundle $V\times \RR^{m}\to V$. We aim to show that $s_{V}$ admits the expression~\eqref{eq:SectionLocalForm}.
Clearly, $s_{V}$ satisfies conditions~\eqref{item:1} and~\eqref{item:2} in Lemma~\ref{lem:Calculus}. Now, since
\[s_{V}^{-1}(0)=\alpha_{0}(U_{0})\subset V\cap (\RR^{d}\times \{0\}^{n-d}),\]
condition~\eqref{item:3} in Lemma~\ref{lem:Calculus} is satisfied. Hence, by Lemma~\ref{lem:Calculus}, the section $s_{V}$ admits the expression~\eqref{eq:SectionLocalForm}, and therefore so does the section $s$ in a neighbourhood of $p$. This completes the proof.
\end{proof}

\subsection{Locality of the Atiyah class in positive amplitude}
In this subsection, we focus on DG manifolds of positive amplitude and their Atiyah classes. 

Recall that the Atiyah class of a DG manifold $(\cM,Q)$ is the obstruction to the existence of an affine connection compatible with $Q$. 
In general, the existence of such connections is a global property rather than a local one.
For example, for a complex manifold $X$, the Atiyah class of the associated DG manifold $(T^{0,1}_{X}[1], \overline{\partial})$ is the obstruction to the existence of a holomorphic connection on the holomorphic tangent bundle $\mathcal{T}_{X}$; see \cite{MR3877426} and \cite{MR86359}. 
Similarly, for an integrable distribution $\cF$ on a manifold $M$, the Atiyah class of the associated DG manifold $(\cF[1], d_{\cF})$ is the obstruction to the existence of projectable connections; see \cite{MR3877426} and \cite{MR0281224}. 
In both cases, although such connections may fail to exist globally, they always exist locally.

In contrast to these examples, which arise from DG manifolds of negative amplitude, the Atiyah class of a DG manifold of positive amplitude encodes local data. We begin by introducing notation to state this more precisely.

Throughout this subsection, $(\cM,Q)$ is a DG manifold of positive amplitude with base $M$.

For each open set $U\subset M$, we denote by $(\cM_{U},Q_{U})$ the DG manifold obtained by restricting $(\cM,Q)$ to $U$. That is, the graded algebra of functions is 
\[\smooth{\cM_{U}}\cong \smooth{U}\tensor_{\smooth{M}}\smooth{\cM},\]
and under this identification, the homological vector field $Q_{U}$ is defined by
 \begin{equation}\label{eq:QU}
  Q_{U}=\id\tensor Q: \smooth{U}\tensor_{\smooth{M}}\smooth{\cM} \to \smooth{U}\tensor_{\smooth{M}}\smooth{\cM}.
 \end{equation}
Indeed, $Q_{U}$ is a derivation on $\smooth{\cM_{U}}$ since, for all $f\in \smooth{M}$, we have $Q(f)=0$ due to degree considerations.

\begin{proposition}\label{prop:Atiyah}
Let $(\cM,Q)$ be a DG manifold of positive amplitude and let $\{U_{i}\}_{i\in I}$ be an open cover of the base manifold $M$. 
Then the Atiyah class $\alpha_{(\cM,Q)}$ of $(\cM,Q)$ is completely determined by the collection $\{\alpha_{i}\}_{i\in I}$, where $\alpha_{i}$ denotes the Atiyah class of $(\cM_{U_{i}},Q_{U_{i}})$. In particular, $\alpha_{(\cM,Q)}=0$ if and only if $\alpha_{i}=0$ for all $i\in I$.
\end{proposition}

We employ the language of sheaves to prove Proposition~\ref{prop:Atiyah}, and begin by recalling the relevant notation.

Let $\cO_{M}$ denote the sheaf of smooth functions on $M$, and let $(\cA,\cQ)$ denote the sheaf of DG $\RR$-algebras associated with $(\cM,Q)$.
Then we have $(\smooth{\cM_{U}},Q_{U})=(\cA(U), \cQ(U))$ for each open subset $U\subset M$.
Since $\cM$ is of positive amplitude, the operator $\cQ$ is $\cO_{M}$-linear. Thus, the pair $(\cA, \cQ)$ is in fact a sheaf of DG $\cO_{M}$-algebras.

Let $(\cE,Q_{\cE})$ denote the sheaf of DG $(\cA,\cQ)$-modules corresponding to a DG vector bundle, denoted simply as $(\cE,Q_{\cE})$ by abuse of notation. 
Clearly, $(\cE,Q_{\cE})$ is a sheaf of DG $\cO_{M}$-modules. That is, $(\cE,Q_{\cE})$ is a cochain complex of sheaves of $\cO_{M}$-modules:
\[\cdots \to \cE^{n-1} \xto{Q_{\cE}} \cE^{n} \xto{Q_{\cE}}\cE^{n+1}\to \cdots,\]
where $\cE^{n}$ denotes the degree $n$ component of $\cE$, and the differential $Q_{\cE}$ is $\cO_{M}$-linear. 
We denote by $Q_{\cE}(U)$ the differential induced on the DG $\cO_{M}(U)$-module $\cE(U)$. 
Since $(\cE,Q_{\cE})$ arises from a DG vector bundle, we have
\begin{equation} \label{eq:cEU}
  \cE(U) \cong \cA(U)\tensor_{\cA(M)}\cE(M), \qquad Q_{\cE}(U) = \cQ(U)\tensor \id + \id \tensor Q_{\cE}(M).
\end{equation}

The following lemma may be known, but we could not find a reference. Note that the sheaf $(\cE,Q_{\cE})$ of DG $\cO_{M}$-modules appearing in the lemma below does not necessarily arise from a DG vector bundle.
\begin{lemma}
  Let $\cO_{M}$ denote the sheaf of algebras of smooth functions on $M$, and let $(\cE,Q_{\cE})$ be a sheaf of DG $\cO_{M}$-modules. Then the cohomology presheaf 
  \[H^{\bullet}(\cE,Q_{\cE}): U\mapsto H^{\bullet}(\cE(U),Q_{\cE}(U)):=\frac{\ker (Q_{\cE}(U))}{\image (Q_{\cE}(U))} \]
   is, in fact, a sheaf of graded $\cO_{M}$-modules.
\end{lemma}

\begin{proof}
  To simplify the notation, we write $\underline{\cE}:=(\cE,Q_{\cE})$ and $\underline{H}:=H(\cE,Q_{\cE})$. We often write $\underline{\cE}^{\bullet}$ and $\underline{H}^{\bullet}$ to emphasise the grading.

  To show that the presheaf $\underline{H}$ is a sheaf, we need to verify the locality and gluing axioms. Since every open cover of $M$ admits a locally finite refinement, it suffices to work with such a cover. 
  Throughout the proof, let $\{U_{i}\}_{i\in I}$ denote a locally finite open cover of $M$, and let $\{\rho_{i}\}_{i\in I}$ denote the partition of unity subordinate to this cover.

  Consider the cochain complex denoted by the symbol $\underline{\cE}^{\bullet}(U_{I})$, defined as the direct product of the family of cochain complexes $\underline{\cE}^{\bullet}(U_{i})$ for $i\in I$: 
  \[ \underline{\cE}^{\bullet}(U_{I}) := \prod_{i\in I} \underline{\cE}^{\bullet}(U_{i}).\]
  The restriction maps $\underline{\cE}(M)\to \underline{\cE}(U_{i})$ for each $i\in I$ induce a cochain map
  \[ F: \underline{\cE}^{\bullet}(M) \to \underline{\cE}^{\bullet}(U_{I}). \]
  The induced map on cohomology is
  \[F_{\ast}: \underline{H}^{\bullet}(M) \to H^{\bullet}(U_{I})\cong \prod_{i\in I} H^{\bullet}(U_{i}).\]
  To verify the locality axiom, it suffices to show that $F_{\ast}$ is injective. 

Define a cochain map $G: \underline{\cE}^{\bullet}(U_{I})\to \underline{\cE}^{\bullet}(M)$ by 
  \[\{s_{i}\}_{i\in I} \mapsto \sum_{i\in I} \rho_{i}\cdot s_{i}\] 
  where $s_{i}\in \cE(U_{i})$. Indeed, $G$ is a cochain map since 
  \[Q_{\cE}(M)\Big(\sum_{i\in I}\rho_{i}\cdot s_{i}\Big) = \sum_{i\in I} Q_{\cE}(M)(\rho_{i}\cdot s_{i}) = \sum_{i\in I} \rho_{i}\cdot Q_{\cE}(U_{i})(s_{i}) .\]
Here, the first equality holds because the sum is locally finite,
and we may invoke the locality axiom of the sheaf $\cE$.

  It is clear that $G\circ F=\id_{\underline{\cE}(M)}$, hence $G_{\ast}\circ F_{\ast}=\id_{\underline{H}(M)}$. Therefore, $F_{\ast}$ is injective.

  It remains to verify the gluing axiom. Consider an element $\{t_{i}\}\in \underline{H}(U_{I})$ such that, for all $i,j \in I$,
  \begin{equation}\label{eq:SheafCondition2}
  t_{i}|_{U_{ij}} = t_{j}|_{U_{ij}}
  \end{equation}
  where $U_{ij}:=U_{i}\cap U_{j}$. To prove the gluing axiom, it suffices to show that there exists $t\in \underline{H}(M)$ such that $F_{\ast}(t)=\{t_{i}\}$.

  For each $i\in I$, let $s_{i}\in \ker Q_{\cE}(U_{i})$ be a representative of $t_{i}$. Then Eq.~\eqref{eq:SheafCondition2} implies that 
  \[s_{i}|_{U_{ij}}=s_{j}|_{U_{ij}} + Q_{\cE}(U_{ij})(\eta_{ij})\]
  for some $\eta_{ij}\in \underline{\cE}(U_{ij})$. Now consider the section
  \[s=\sum_{i\in I} \rho_{i}\cdot s_{i}\in \underline{\cE}(M).\] 
  For each $j\in I$, we compute
  \begin{equation*}
    s|_{U_{j}} = \sum_{i\in I} \rho_{i}|_{U_{j}}\cdot s_{i}|_{U_{ij}}
    = \sum_{i\in I} \rho_{i}|_{U_{j}}\cdot (s_{j}|_{U_{ij}} + Q_{\cE}(U_{ij})(\eta_{ij}))
    = s_{j} + Q_{\cE}(U_{j})(\sum_{i\in I}\rho_{i}|_{U_{j}}\cdot \eta_{ij}).
  \end{equation*} 
  In other words,
  \[F(s)\in \{s_{i}\} + \image Q_{\cE}(U_{I}).\] 
Recall that $G$ is a cochain map and 
that $G\circ F=\id_{\underline{\cE}(M)}$. 
By the construction, we have $F(s)\in \ker Q_{\cE}(U_{I})$ and thus, $s=GF(s)\in \ker Q_{\cE}(M)$.
Define $t:=[s] \in \underline{H}(M)$.
Then $F_{\ast}(t)=\{t_{i}\}$ establishes the gluing property. This completes the proof.
\end{proof}

When $(\cE,Q_{\cE})$ arises from a DG vector bundle, we have the following corollary. 
Note that the symbol $(\sections{\cM_{U}; \cE}, (Q_{\cE})_{U})$ in the corollary below denotes the DG $(\smooth{\cM_{U}},Q_{U})$-module of sections of the restricted DG vector bundle $(\cE|_{\cM_{U}},(Q_{\cE})_{U})\to (\cM_{U},Q_{U})$. In terms of sheaves, it corresponds to $(\cE(U), Q_{\cE}(U))$ in Eq.~\eqref{eq:cEU}. 

\begin{corollary}\label{cor:VBSheaf}
Let $(\cM,Q)$ be a DG manifold of positive amplitude with base manifold $M$. Given a DG vector bundle $(\cE,Q_{\cE})$, the presheaf defined as the cohomology of the DG module of sections
\[U\mapsto H^{\bullet}(\sections{\cM_{U}; \cE}, (Q_{\cE})_{U})\]
is a sheaf of graded $\cO_{M}$-modules.
\end{corollary}

Before proving Proposition~\ref{prop:Atiyah}, consider the tangent bundle $(\tangent{\cM},\liederivative{Q})$. 
As in classical differential geometry, for each open subset $U\subset M$, the inclusion of graded manifolds $\cM_{U}\into \cM$ yields an algebra homomorphism
\[\smooth{\cM}\to \smooth{\cM_{U}}\cong \smooth{U}\tensor_{\smooth{M}}\smooth{\cM}\] 
defined by the assignment
\[f\mapsto f|_{U}=1_{U}\tensor f ,\]
where $1_{U}$ denotes the constant function equal to $1$ on $U$. 

\begin{lemma}\label{lem:XXU}
There exists a unique restriction map on vector fields
\[
  \XX(\cM) \to \XX(\cM_{U}), \quad X\mapsto X_{U}
\]
satisfying
\[ X_{U}(f|_{U})=X(f)|_{U}.\]
Moreover, this map is a morphism of DG Lie algebras:
\[(\XX(\cM), \liederivative{Q}, [\argument,\argument]) \to (\XX(\cM_{U}), \liederivative{Q_{U}}, [\argument,\argument]) ,\]
where the Lie algebra structures are given by the graded commutator brackets.
\end{lemma}

\begin{proof}
Note that derivations are local operators: if $f|_{U}=g|_{U}$ then $X(f)|_{U}=X(g)|_{U}$ for any $X\in \XX(\cM)$. This ensures that $X_{U}$ is well-defined on the subspace of functions that extend to global functions.
When $U$ is small enough, $\cM_{U}$ admits coordinate functions which extend to global functions on $\cM$; since a derivation is uniquely determined by its action on coordinates, $X_{U}$ is uniquely determined. For a general open set $U$, a standard argument using a partition of unity proves the existence of the restriction map.
  
It is standard that $(\XX(\cM),\liederivative{Q}, [\argument,\argument])$ and $(\XX(\cM_{U}),\liederivative{Q_{U}}, [\argument,\argument])$ are DG Lie algebras.

  To check that the map preserves brackets, let $X,Y\in \XX(\cM)$ be homogeneous. Then for any $f\in \smooth{\cM}$,
  \begin{multline*}
    [X,Y](f)|_{U}=XY(f)|_{U}-(-1)^{\degree{X}\degree{Y}} YX(f)|_{U}\\
    =X_{U}(Y(f)|_{U})-(-1)^{\degree{X}\degree{Y}} Y_{U}(X(f)|_{U}) = [X_{U},Y_{U}](f|_{U}).\end{multline*}
    Hence, $[X,Y]_{U}=[X_{U},Y_{U}]$, and the map preserves the Lie bracket.

  Finally, the homological vector field $Q_{U}$ in Eq.~\eqref{eq:QU} satisfies 
  \[Q_{U}(f|_{U})=(\id \tensor Q)(1_{U}\tensor f)=1_{U} \tensor Q(f)=Q(f)|_{U},\]
  so $Q_{U}$ defined via Eq.~\eqref{eq:QU} agrees with the restriction of $Q$ as a vector field. Therefore, the differential is preserved: 
  \[ ( \liederivative{Q}(X))_{U}=\liederivative{Q_{U}}(X_{U}).\]
  This shows that the restriction map defines a morphism of DG Lie algebras.
\end{proof}

Now we are ready to prove Proposition~\ref{prop:Atiyah}.
\begin{proof}[Proof of Proposition~\ref{prop:Atiyah}]
  Let $\nabla$ be an affine connection on $\cM$. Given an open subset $U\subset M$, let $\nabla^{U}$ be the pullback connection on $\cM_{U}$ along the inclusion $\cM_{U}\into \cM$. That is, $\nabla^{U}:\XX(\cM_{U})\times \XX(\cM_{U})\to \XX(\cM_{U})$ is defined by
  \[\nabla^{U}_{X_{U}} Y_{U} := (\nabla_{X}Y)_{U}\]
  for $X,Y\in \XX(\cM)$. The well-definedness of $\nabla^{U}$ follows from the same reasoning as in Lemma~\ref{lem:XXU}. 

  Now, for homogeneous $X,Y\in \XX(\cM)$, we compute:
\begin{multline*}
  \big(\At^{\nabla}(X,Y)\big)_{U} = [Q,\nabla_{X}Y]_{U} - (\nabla_{[Q,X]}Y)_{U} -(-1)^{\degree{X}} (\nabla_{X}[Q,Y])_{U}\\
  =[Q_{U},(\nabla^{U}_{X_{U}}Y_{U})]-\nabla^{U}_{[Q_{U},X_{U}]}Y_{U}-(-1)^{\degree{X_{U}}}\nabla^{U}_{X_{U}}[Q_{U},Y_{U}]
  =\At^{\nabla^{U}}(X_{U}, Y_{U}).
\end{multline*}
 
  Let $(\cE,Q_{\cE})=(\cotangent{\cM}\tensor \End(\tangent{\cM}),\liederivative{Q})$.
  The computation above shows that the restriction map
  \[H^{\bullet}(\sections{\cM;\cE},Q_{\cE})\to H^{\bullet}(\sections{\cM_{U}; \cE}, (Q_{\cE})_{U})\]
   sends the Atiyah class $\alpha_{(\cM,Q)}$ of $(\cM,Q)$ to the Atiyah class $\alpha_{(\cM_{U},Q_{U})}$ of $(\cM_{U},Q_{U})$. 
   
   By the gluing axiom, Corollary~\ref{cor:VBSheaf} implies that if $\{U_{i}\}$ is an open cover of $M$, then the collection
   $\{\alpha_{(\cM_{U_{i}},Q_{U_{i}})}\}_{i}$ uniquely determines $\alpha_{(\cM,Q)}$. Moreover, by the locality axiom, $\alpha_{(\cM,Q)}=0$ if and only if $\alpha_{(\cM_{U_{i}},Q_{U_{i}})}=0$ for all $i$. 
   
   This completes the proof of Proposition~\ref{prop:Atiyah}. 
\end{proof}

\subsection{The Atiyah class of DG manifolds of amplitude $+1$}\label{subsec:AtiyahClass}

In this subsection, we focus on DG manifolds of amplitude $+1$ and give an explicit computation of their Atiyah classes. 

For the remainder of this section, $\pi:E\to M$ is a vector bundle, and $s:M\to E$ is a section. Also, we always assume that $(\cM,Q)=(E[-1],\iota_{s})$.

\subsubsection{Vector fields and affine connections}
Recall that for a vector bundle $\pi:E\to M$, the $\smooth{E}$-module of vector fields $\XX(E)$ on the total space $E$ fits into the short exact sequence
\[0\to \sections{\pi^{\ast}E}\to \XX(E)\xto{\pi_{\ast}} \sections{\pi^{\ast}\tangent{M}} \to 0\]
of $\smooth{E}$-modules. 
Analogously, for the graded manifold $\cM=E[-1]$ with base $M$, one obtains the corresponding short exact sequence \begin{equation}\label{eq:VFSES}
0\to \sections{S(E[-1]^{\vee})} \tensor_{\smooth{M}}\sections{E[-1]} \xto{\iota} \XX(E[-1]) \xto{(\pi[1])_{\ast}} \sections{S(E[-1]^{\vee})} \tensor_{\smooth{M}}\XX(M)\to 0
\end{equation}
of graded $\smooth{\cM}=\sections{S(E[-1]^{\vee})}$-modules. Here, the map $\iota$ denotes the graded $\sections{S(E[-1]^{\vee})}$-linear extension of the interior product associated with an element of $\sections{E[-1]}$. The map $(\pi[1])_{\ast}$ is the pushforward induced by the morphism of DG manifolds $\pi[1]: E[-1]\to M$.  The map $\pi[1]$ can be understood as the composition of the degree shift $E[-1]\to E$ with the bundle projection $\pi:E\to M$.

A $\sections{S(E[-1]^{\vee})}$-linear map 
\[\tau: \sections{S(E[-1]^{\vee})} \tensor_{\smooth{M}}\XX(M) \to \XX(E[-1])\]
satisfying $(\pi[1])_{\ast}\circ \tau = \id$ is called a (graded) horizontal lift of vector fields from $M$ to $\cM=E[-1]$.

Let $\nabla^{E}$ be a $\tangent{M}$-connection on $E$. Extending its dual connection via the Leibniz rule yields a $\tangent{M}$-connection on $\sections{\Lambda E^{\vee}} \cong \sections{S(E[-1]^{\vee})}$. By abuse of notation, we again denote this connection by $\nabla^{E}$. The horizontal lift associated with $\nabla^{E}$ is the graded $\sections{S(E[-1]^{\vee})}$-linear map
\[\tau=\tau^{\nabla^{E}}:\sections{S(E[-1]^{\vee})} \tensor_{\smooth{M}}\XX(M) \to \XX(E[-1]) \cong \Der(\sections{S(E[-1]^{\vee})})\]
characterised by 
\[\tau(X): \xi \mapsto \nabla^{E}_{X}\xi\]
for $X\in \XX(M)$ and $\xi \in \sections{S(E[-1]^{\vee})}$. In particular, we have $\nabla^{E}_{X}f=X(f)$ for $X\in \XX(M)$ and $f\in \smooth{M}$, which implies $(\pi[1])_{\ast}\circ \tau(X)=X$ for all $X\in \XX(M)$. This proves the following lemma (see also~\cite[Proof of Lemma~4.10]{schatz:phd}).

\begin{lemma}\label{lem:grVF}
Upon a choice of a $\tangent{M}$-connection $\nabla^{E}$ on $E$, there exists a $\sections{S(E[-1]^{\vee})}$-module isomorphism
\[
\sections{S(E[-1]^{\vee})} \tensor_{\smooth{M}}(\sections{E[-1]}\oplus \sections{\tangent{M}}) \cong \XX(E[-1])
\]
given by the $\sections{S(E[-1]^{\vee})}$-linear extension of the assignments
\[ a\mapsto \iota_{a}, \qquad \text{and} \qquad X\mapsto \tau(X)=\nabla^{E}_{X}\]
for $a\in \sections{E[-1]}$ and $X\in \XX(M)$. 
\end{lemma}

\begin{remark}
  A $\tangent{M}$-connection $\nabla^{E}$ on $E$ naturally induces a corresponding $\tangent{M}$-connection on the graded vector bundle $E[-1]$, and vice versa. In fact, such identification was already made in Lemma~\ref{lem:grVF}. In what follows, we will freely identify these connections depending on the context.
\end{remark}

As a consequence of Lemma~\ref{lem:grVF}, any graded vector field on $\cM=E[-1]$ has degree at most $+1$. Moreover, by restricting to the homogeneous components of degree $+1$ and $0$, we have
\begin{equation}\label{eq:grVFdegree}
\XX(E[-1])^{+1} \cong \sections{E[-1]}, \qquad
\XX(E[-1])^{0} \cong \sections{E[-1]^{\vee} \tensor  E[-1]} \oplus \XX(M)
\end{equation}
where the symbol $\XX(E[-1])^{n}$ denotes the homogeneous component of degree $n$.

Next, we compute the Lie bracket on $\XX(E[-1])$. To simplify the notation, we denote the horizontal lift of $X\in \XX(M)$, obtained in Lemma~\ref{lem:grVF}, by $\hat{X}:=\tau(X)\in \XX(E[-1])$.
\begin{lemma}\label{lem:LieBracket}
Under the decomposition in Lemma~\ref{lem:grVF}, the graded commutator bracket $[\argument,\argument]$ on $\XX(E[-1])$ satisfies
\[
 [\iota_{a},\iota_{b}]=0, \quad [\hat{X}, \iota_{a}]=\iota_{\nabla^{E}_{X}a}, \quad [\hat{X},\hat{Y}]=\widehat{[X,Y]}-\iota_{R^{\nabla^{E}}(X,Y)}
\]
where $R^{\nabla^{E}}$ is the curvature of $\nabla^{E}$, and $\iota_{R^{\nabla^{E}}(X,Y)}$ denotes the image of $R^{\nabla^{E}}(X,Y)$, viewed as an element of $\sections{E[-1]^{\vee}}\tensor_{\smooth{M}}\sections{E[-1]}$, under the map $\iota$ in the short exact sequence~\eqref{eq:VFSES}.
\end{lemma}

\begin{proof}
The result follows by the direct computation using the definition of $\iota$ and $\tau=\tau^{\nabla^{E}}$.
\end{proof}

Finally, we investigate torsion-free affine connections on $\cM=E[-1]$.

\begin{proposition}\label{prop:Connection}
A torsion-free affine connection $\nabla:\XX(E[-1])\times \XX(E[-1]) \to \XX(E[-1])$ is equivalent to a triple $(\nabla^{E}, \nabla^{M}, \beta)$ where 
\begin{enumerate}
\item $\nabla^{E}:\sections{\tangent{M}}\times \sections{E}\to \sections{E}$ is a $\tangent{M}$-connection on $E$;
\item $\nabla^{M}:\sections{\tangent{M}}\times \sections{\tangent{M}}\to \sections{\tangent{M}}$ is a torsion-free affine connection on $M$;
\item an element $\beta\in \sections{\cotangent{M}\tensor \cotangent{M}\tensor \End(E)}$ satisfying
\[\beta(X,Y)-\beta(Y,X)=- R^{\nabla^{E}}(X,Y).\]
\end{enumerate}
In particular, denoting $\hat{X}=\nabla^{E}_{X}\in \XX(E[-1])$ for each $X\in \XX(M)$, we have explicit formulae:
\[
 \nabla_{\iota_{a}}\iota_{b}=0, 
\quad \nabla_{\iota_{a}} \hat{X}=0, 
\quad \nabla_{\hat{X}}\iota_{a}=\iota_{\nabla^{E}_{X}a}, 
\quad \nabla_{\hat{X}}\hat{Y}=\widehat{\nabla^{M}_{X}Y}+\iota_{\beta(X,Y)},
\]
where $\iota_{\beta(X,Y)}$ denotes the image of $\beta(X,Y)$, viewed as an element of $\sections{E[-1]^{\vee}}\tensor_{\smooth{M}}\sections{E[-1]}$, under the map $\iota$ in the short exact sequence~\eqref{eq:VFSES}.
\end{proposition}

\begin{remark}
A similar result also holds for graded manifolds of amplitude $-1$. It is partly shown in \cite{arXiv:2302.04956} and is related to~\cite{MR4523146}.
\end{remark}

\begin{proof}
  Suppose we are given a triple $(\nabla^{E},\nabla^{M},\beta)$. Then, under the identification in Lemma~\ref{lem:grVF}, it is straightforward to verify that the explicit formulae in Proposition~\ref{prop:Connection} uniquely determine an affine connection $\nabla$ on $\cM=E[-1]$. Moreover, by Lemma~\ref{lem:LieBracket}, it is clear that this connection $\nabla$ is torsion-free.

Conversely, let $\nabla$ be a torsion-free affine connection on $E[-1]$. Using Lemma~\ref{lem:grVF} and the torsion-freeness, the connection $\nabla$ is completely determined by its values on pairs of the form 
\[\nabla_{\iota_{a}}\iota_{b},\quad \nabla_{\hat{X}}\iota_{a},\quad \nabla_{\hat{X}}\hat{Y}\] for $a,b\in\sections{E[-1]}$ and $X,Y\in \XX(M)$.

First, observe that for any $a,b\in \sections{E[-1]}$, the vector field $\nabla_{\iota_{a}}\iota_{b}$ is of degree $2$. By Lemma~\ref{lem:grVF}, any homogeneous vector field of degree $2$ on $E[-1]$ must vanish. Thus, $\nabla_{\iota_{a}}\iota_{b}=0$.

Second, consider the assignment $\XX(M)\times \sections{E[-1]}\to \sections{E[-1]}$ defined by
\[(X,a)\mapsto \nabla_{\hat{X}}\iota_{a}\in \XX(E[-1])^{+1} \cong \sections{E[-1]}.\]
Since $\hat{X}(f)=X(f)$ for $f\in \smooth{M}$, this assignment satisfies 
 \[ (fX,a)\mapsto f\nabla_{\hat{X}}\iota_{a}, \qquad (X,fa)\mapsto X(f)\cdot \iota_{a}+f \nabla_{\hat{X}}\iota_{a}.\]
Thus, this defines a $\tangent{M}$-connection on $E[-1]$, or equivalently, a $\tangent{M}$-connection $\nabla^{E}$ on $E$. In other words, $\nabla$ induces a connection $\nabla^{E}$ such that $\nabla_{\hat{X}}\iota_{a}=\iota_{\nabla^{E}_{X}a}$.
 
Moreover, the connection $\nabla^{E}$ is independent of the choice of horizontal lifts. 
Denote another horizontal lift of $X\in \XX(M)$ by $\bar{X}$. 
By the short exact sequence \eqref{eq:VFSES}, we have $\hat{X}-\bar{X}\in \image \iota$. 
Since $\nabla_{\iota_{b}}\iota_{a}=0$ for all $a,b\in \sections{E[-1]}$, it follows that $\nabla_{\hat{X}-\bar{X}}\iota_{a}=0$ for all $a\in \sections{E[-1]}$. 
This shows that the induced connection $\nabla^{E}$ is independent of the choice of horizontal lifts.

Similarly, the assignment 
\[(X,Y)\mapsto (\pi[1])_{\ast}(\nabla_{\hat{X}}\hat{Y})\in \sections{\tangent{M}}\]
defines an affine connection $\nabla^{M}$ on $M$, where $(\pi[1])_{\ast}$ is as defined in the short exact sequence~\eqref{eq:VFSES}. A similar argument shows that $\nabla^{M}$ is also independent of the choice of horizontal lifts. 

So far, we have shown that a torsion-free affine connection $\nabla$ on $E[-1]$ induces $\nabla^{E}$ and $\nabla^{M}$, and that both $\nabla^{E}$ and $\nabla^{M}$ are independent of the choice of horizontal lift.
We may therefore choose the horizontal lift to be the one associated with $\nabla^{E}$, i.e., $\hat{X}=\nabla^{E}_{X}$. In this setting, define 
\[\iota_{\beta(X,Y)}:=\nabla_{\hat{X}}\hat{Y}-\widehat{\nabla^{M}_{X}Y}.\]

Finally, it is straightforward to check that, by Lemma~\ref{lem:LieBracket} and torsion-freeness of $\nabla$, the connection $\nabla^{M}$ is torsion-free and that  
\[\beta(X,Y)-\beta(Y,X)=-R^{\nabla^{E}}(X,Y).\]
This completes the proof.
\end{proof}

\subsubsection{Computation of the Atiyah class: global} \label{subsec:global}
In general, the Atiyah class of a DG manifold $(\cM,Q)$ lies in the first cohomology group of the cochain complex 
$(\sections{\cotangent{\cM}\tensor \End (\tangent{\cM})}^{\bullet}, \liederivative{Q})$.
In fact, by Remark~\ref{rem:AtiyahProperty}, it is given by an element in the first cohomology group of the subcomplex 
\[(\sections{\Hom(S^{2}(\tangent{\cM}), \tangent{\cM})}^{\bullet}, \liederivative{Q}).\]
In this subsection, we work with the DG manifold $(\cM,Q)=(E[-1],\iota_{s})$, and compute its Atiyah class by analysing the first cohomology group of this complex.

Choose a $\tangent{M}$-connection $\nabla^{E}$ on $E$. Then, by Lemma~\ref{lem:grVF}, one obtains an isomorphism of graded $\smooth{\cM}=\sections{S(E[-1]^{\vee})}$-modules
\begin{equation*}\label{eq:1-2-tensor}
\sections{\Hom(S^{2}(\tangent{\cM}), \tangent{\cM})} \cong  \sections{S(E[-1]^{\vee}) \tensor \Hom (S^{2}(E[-1]\oplus \tangent{M}),E[-1]\oplus \tangent{M})}.
\end{equation*}

Under this identification, a straightforward degree argument shows that:
\begin{enumerate}
  \item $\sections{\Hom(S^{2}(\tangent{\cM}), \tangent{\cM})}$ is concentrated in degree $\leq 1$;
  \item the degree $+1$ component of $\sections{\Hom(S^{2}(\tangent{\cM}), \tangent{\cM})}$ is 
    \begin{equation}\label{eq:1-2-1}
    \sections{\Hom(S^{2}(\tangent{\cM}), \tangent{\cM})}^{1}\cong \sections{\Hom(S^{2}(\tangent{M}), E[-1])} ;
    \end{equation}
  \item the degree $0$ component of $\sections{\Hom(S^{2}(\tangent{\cM}), \tangent{\cM})}$ is
    \begin{equation}\label{eq:1-2-0}
      \begin{aligned}
      \sections{\Hom(S^{2}(\tangent{\cM}), \tangent{\cM})}^{0}\cong 
      & \quad \sections{E[-1]^{\vee}\tensor\Hom(S^{2}(\tangent{M}), E[-1])}\\
      & \oplus \sections{\Hom(E[-1]\tensor \tangent{M}, E[-1])} \\
      & \oplus \sections{\Hom(S^{2}(\tangent{M}), \tangent{M})} .
      \end{aligned}
    \end{equation}
\end{enumerate}

Next, we study the restriction of the coboundary operator to the degree~$0$ component:
\[
\liederivative{Q}:\sections{\Hom(S^{2}(\tangent{\cM}), \tangent{\cM})}^{0}\to \sections{\Hom(S^{2}(\tangent{\cM}), \tangent{\cM})}^{1} .
\]
Since $Q=\iota_{s}$ is $\smooth{M}$-linear, the operator $\liederivative{Q}$ is also $\smooth{M}$-linear.
\begin{lemma}\label{lem:DecomposeLQ}
The above coboundary operator decomposes as a sum of $\smooth{M}$-linear maps 
\[\liederivative{Q}=d_{1}+d_{2}+d_{3}\] 
where:
\begin{enumerate}
\item the first map 
  \[d_{1}: \sections{E[-1]^{\vee}\tensor\Hom(S^{2}(\tangent{M}), E[-1])} \to \sections{\Hom(S^{2}(\tangent{M}), E[-1])}\] 
  is given by applying $\iota_{s}:\sections{E[-1]^{\vee}}\to \smooth{M}$ to the first factor.
\item the second map 
  \[d_{2}: \sections{\Hom(E[-1]\tensor \tangent{M}, E[-1])} \to \sections{\Hom(S^{2}(\tangent{M}), E[-1])}\]
is induced by precomposition with the bundle map
\[S^{2}(\tangent{M}) \to E[-1]\tensor \tangent{M}\] 
defined by 
\[X\odot Y \mapsto \nabla^{E}_{X}s\tensor Y + \nabla^{E}_{Y}s \tensor X\]
for $X,Y\in \XX(M)$.
\item the third map 
\[d_{3}:\sections{\Hom(S^{2}(\tangent{M}), \tangent{M})} \to \sections{\Hom(S^{2}(\tangent{M}), E[-1])}\] 
is induced by composition with the bundle map 
\[\nabla^{E}s:\tangent{M}\to E[-1]\] 
defined by 
\[ X\mapsto \nabla^{E}_{X}s\] 
for $X\in \XX(M)$.
\end{enumerate}
\end{lemma}

\begin{proof}
  It follows from a direct and explicit computation of $[\iota_{s},\argument]$ in Lemma~\ref{lem:LieBracket}, together with identifications~\eqref{eq:1-2-1} and~\eqref{eq:1-2-0}. 
\end{proof}

\begin{corollary}\label{cor:H1}
  For $(\cM,Q)=(E[-1],\iota_{s})$, the first cohomology group $H^{1}(\sections{\Hom(S^{2}(\tangent{\cM}), \tangent{\cM})}, \liederivative{Q}) $ is isomorphic to the quotient
  \[  \left. \sections{\Hom(S^{2}(\tangent{M}), E[-1])}  \right/ \langle\image(d_{1}), \image(d_{2}), \image(d_{3}) \rangle \]
  where $d_{1}$, $d_{2}$ and $d_{3}$ are as defined in Lemma~\ref{lem:DecomposeLQ}.
\end{corollary}

In particular, if $s\in \sections{E}$ is nowhere vanishing, the Atiyah class of $(E[-1],\iota_{s})$ vanishes.

\begin{corollary}\label{cor:NotLocus}
Assume that $s\in \sections{E}$ is a nowhere vanishing section. Then, for $(\cM,Q)=(E[-1],\iota_{s})$, we have 
\[H^{1}(\sections{\Hom(S^{2}(\tangent{\cM}), \tangent{\cM})}, \liederivative{Q})=0.\]
In particular, the Atiyah class of $(E[-1],\iota_{s})$ vanishes.
\end{corollary}

\begin{proof}
Since $s$ is nowhere vanishing, there exists a dual section $\xi \in \sections{E^{\vee}}$ such that $\iota_{s}(\xi) = 1 \in \smooth{M}$. Thus, by Lemma~\ref{lem:DecomposeLQ}, we have 
\[\image(d_{1}) = \sections{\Hom(S^{2}(\tangent{M}), E[-1])}.\] 
By Corollary~\ref{cor:H1}, the first cohomology group vanishes, and thus the Atiyah class, which belongs to the first cohomology group, vanishes. This completes the proof.
\end{proof}

Geometrically, Corollary~\ref{cor:NotLocus} implies that the Atiyah class of $(E[-1],\iota_{s})$ vanishes if the section $s:M\to E$ does not intersect the zero section $\sigma_{0}:M\to E$.

We now compute an Atiyah cocycle of the DG manifold $(E[-1],\iota_{s})$. 

\begin{proposition}\label{prop:AtiyahCocycleGlobal}
Let $(\cM,Q)=(E[-1],\iota_{s})$ be a DG manifold. Given a torsion-free affine connection $\nabla$ on $\cM$, the associated Atiyah cocycle $\At^{\nabla}\in \sections{\Hom(S^{2}(\tangent{\cM}), \tangent{\cM})}$ is completely determined by
\[\At^{\nabla}(\hat{X},\hat{Y})=\iota_{\nabla^{E}_{X}\nabla^{E}_{Y}s-\nabla^{E}_{(\nabla^{M}_{X}Y)}s + \beta(X,Y)s}\]
for $X,Y\in \XX(M)$ where $\nabla = ( \nabla^{E},\nabla^{M}, \beta)$ as in Proposition~\ref{prop:Connection}, and $\hat{X} = \nabla^{E}_{X}\in \XX(\cM)$ denotes the horizontal lift of $X\in \XX(M)$ associated with $\nabla^{E}$.
\end{proposition}
\begin{remark}
Under the identification~\eqref{eq:1-2-1}, the Atiyah cocycle $\At^{\nabla}$ above corresponds to the bundle map $\operatorname{at}^{\nabla}:S^{2}(\tangent{M}) \to E[-1]$ defined by
$\operatorname{at}^{\nabla}(X,Y)=\nabla^{E}_{X}\nabla^{E}_{Y}s - \nabla^{E}_{(\nabla^{M}_{X}Y)}s + \beta(X,Y)s.$
\end{remark}
\begin{proof}
Since $\At^{\nabla}$ is $\smooth{\cM}$-linear in both arguments, it suffices to compute 
\[\At^{\nabla}(\iota_{a},\iota_{b}), \quad \At^{\nabla}(\iota_{a},\hat{X}), \quad \At^{\nabla}(\hat{X}, \hat{Y})\]
for $a,b\in \sections{E[-1]}$ and $X,Y\in \XX(M)$. A degree argument shows that $\At^{\nabla}(\iota_{a},\iota_{b})=0$ and  $\At^{\nabla}(\iota_{a},\hat{X})=0$.

By Proposition~\ref{prop:Connection} and Lemma~\ref{lem:LieBracket}, we have
\begin{align*}
\At^{\nabla}(\hat{X},\hat{Y}) &=[Q,\nabla_{\hat{X}}\hat{Y}]-\nabla_{[Q,\hat{X}]}\hat{Y} - \nabla_{\hat{X}}[Q,\hat{Y}]\\
&=[\iota_{s}, \widehat{\nabla^{M}_{X}Y}+ \iota_{\beta(X,Y)}] + \nabla_{\iota_{\nabla^{E}_{X}s}}\hat{Y} + \nabla_{\hat{X}}(\iota_{\nabla^{E}_{Y}s})\\
&=\iota_{ \beta(X,Y)s- \nabla^{E}_{(\nabla^{M}_{X}Y)}s} +0 + \iota_{\nabla^{E}_{X}\nabla^{E}_{Y}s}.
\end{align*}
 This concludes the proof.
\end{proof}

\subsubsection{Computation of the Atiyah class: local}\label{subsec:local}
We now compute the Atiyah class of $(E[-1],\iota_{s})$ in local coordinates. 
By Proposition~\ref{prop:Atiyah}, this completely determines the Atiyah class globally. 

We assume that $M=\RR^{n}$ and $E=\RR^{n}\times \RR^{m}$. 
We denote the coordinate functions on $M$ by $\{x^{1},\ldots, x^{n}\}$, the corresponding vector fields by $\{\partial_{1},\ldots,\partial_{n}\}$, and the corresponding $1$-forms by $\{dx^{1},\ldots, dx^{n}\}$. Moreover, $\{\epsilon_{1},\ldots, \epsilon_{m}\}$ and $\{\xi^{1},\ldots, \xi^{m}\}$ denote the standard frames of $E[-1]$ and $E[-1]^{\vee}$, respectively. We define the section $s\in \sections{E}$ by 
\[s=s^{1} \epsilon_{1}+ \cdots + s^{m} \epsilon_{m}\] 
where $s^{i}\in \smooth{M}$.

In terms of the graded manifold $E[-1]$, the coordinate functions are $x^{1},\ldots,x^{n},\xi^{1},\ldots,\xi^{m}$, where $x^{i}$ is of degree $0$ and $\xi^{i}$ is of degree $-1$. Then, $\{\partial_{1},\ldots, \partial_{n}, \epsilon_{1},\ldots,\epsilon_{m}\}$ can be viewed as a basis of the graded $\smooth{E[-1]}$-module $\XX(E[-1])\cong \Der(\smooth{E[-1]})$ in such a way that
\[\partial_{i}(x^{k})=\delta_{i}^{k}, \quad \partial_{i}(\xi^{k})=0, \quad \epsilon_{j}(x^{k})=0, \quad \epsilon_{j}(\xi^{k})=\delta_{j}^{k}\]
where $\delta_{i}^{j}$ is the Kronecker delta. Note that the vector field $\partial_{i}$ is of degree $0$ and the vector field $\epsilon_{j}$ is of degree $1$. 
We denote the dual basis of differential $1$-forms $\Omega(E[-1])$ by $\{dx^{1},\ldots,dx^{n},d\xi^{1},\ldots d\xi^{m}\}$. 
The degree of $dx^{i}$ is $0$ and that of $d\xi^{i}$ is $-1$.

\begin{remark}\label{rem:Free}
  In this local setting, it is clear from Eqs.~\eqref{eq:1-2-1} and~\eqref{eq:1-2-0} that the degree $1$ and $0$ (in fact, all degrees) components of $\sections{\Hom(S^{2}(\tangent{\cM}),\tangent{\cM})}$ are both free $\smooth{M}$-modules. In particular, they admit an ordered basis.
\end{remark}

Consider the trivial affine connection $\nabla$ on $E[-1]$. It is straightforward to check that $\nabla$ is torsion-free. In terms of Proposition~\ref{prop:Connection}, $\nabla$ corresponds to the triple $(\nabla^{E},\nabla^{M}, \beta)$, where $\nabla^{E}$ and $\nabla^{M}$ are trivial connections and $\beta=0$. 

In this setting, the following lemma immediately follows from Lemma~\ref{lem:DecomposeLQ}, Corollary~\ref{cor:H1} and Proposition~\ref{prop:AtiyahCocycleGlobal}.
Note that in the lemma below, we use the identification
\[\sections{\Hom(S^{2}(\tangent{\cM}),\tangent{\cM})}\cong \sections{S^{2}(\cotangent{\cM})\tensor \tangent{\cM}}\]
for the graded manifold $\cM=E[-1]$.

\begin{lemma}\label{lem:AtLocal}
In the above setting, the coboundary operator $\liederivative{Q}=d_{1}+d_{2}+d_{3}$ in Lemma~\ref{lem:DecomposeLQ} satisfies
\begin{gather}
d_{1}(\xi^{l} \cdot dx^{i}\odot dx^{j}\tensor \epsilon_{k})=s^{l}\cdot dx^{i}\odot dx^{j}\tensor \epsilon_{k}, \label{eq:d1} \\
d_{2}(d\xi^{i}\tensor dx^{j}\tensor \epsilon_{k})=\sum_{p=1}^{n} \frac{\partial s^{i}}{\partial x^{p}} \cdot dx^{p}\odot dx^{j}\tensor \epsilon_{k},  \label{eq:d2}\\
d_{3}(dx^{i}\odot dx^{j}\tensor \partial_{k})=\sum_{p=1}^{m} \frac{\partial s^{p}}{\partial x^{k}}\cdot dx^{i}\odot dx^{j}\tensor \epsilon_{p}, \label{eq:d3}
\end{gather}
for all $i,j,k$,
and the element
\begin{equation}\label{eq:AtLocal}
\At^{\nabla}=\At=\sum_{i\leq j}\sum_{k}\frac{\partial^{2}s^{k}}{\partial x^{i}\partial x^{j}} \cdot dx^{i}\odot dx^{j}\tensor \epsilon_{k} \in \sections{S^{2}(\cotangent{M}) \tensor E[-1]} 
\end{equation}
represents the Atiyah class in the first cohomology group
\[H^{1}(\sections{S^{2}(\cotangent{\cM}) \tensor \tangent{\cM}}, \liederivative{Q}) \cong  \left. \sections{S^{2}(\cotangent{M}) \tensor E[-1]}  \right/ \langle\image(d_{1}), \image(d_{2}), \image(d_{3}) \rangle . \]
\end{lemma}

We point out that, under the identification $\sections{S^{2}(\cotangent{M}) \tensor E[-1]}\cong \sections{\Hom(S^{2}(\tangent{M}), E[-1])}$, Eq.~\eqref{eq:AtLocal} is equivalent to
\[\At^{\nabla}(\partial_{i}, \partial_{j})=\sum_{k}\frac{\partial^{2}s^{k}}{\partial x^{i}\partial x^{j}} \cdot \epsilon_{k}.\]

\section{Proof of Theorem~\ref{thm:Main}} \label{sec:Proof}
In this section, we prove Theorem~\ref{thm:Main}, which characterises the Atiyah class of the DG manifold $(E[-1],\iota_{s})$ in terms of the clean intersection of the section $s$ with the zero section.

By Proposition~\ref{prop:CleanIntersection}, Proposition~\ref{prop:Atiyah} and Corollary~\ref{cor:NotLocus}, it suffices to prove the following lemma to establish Theorem~\ref{thm:Main}.

\begin{lemma}\label{lem:SubMain}
Let $E\to M$ be a vector bundle, and let $s\in \sections{E}$ be a section.
Assume $\dim M = n$ and $\rank E = m$. The following are equivalent.
\begin{enumerate}[label=(\Alph*)]
\item \label{item:A} For each $p\in s^{-1}(0)$, there exists a neighbourhood $U\subset M$ of $p$ such that the Atiyah class of the DG manifold $(E|_{U}[-1],\iota_{s|_{U}})$ vanishes.
\item \label{item:B} For each $p\in s^{-1}(0)$, there exists a neighbourhood $U \subset M$ of $p$, a local coordinate chart $\phi:U\to \RR^{n}$, and a local frame $\{e_{1}, \ldots, e_{m}\}$ of $E$ over $U$ such that 
\[
s\circ \phi^{-1}(x_{1},\ldots, x_{n}) = x_{1} e_{1} + \cdots + x_{r} e_{r}
\]
for some $r$. 
\end{enumerate}
\end{lemma}

The remainder of this section is devoted to proving Lemma~\ref{lem:SubMain}. One direction is straightforward.

\begin{proof}[Proof of \ref{item:B} $\Rightarrow$ \ref{item:A}]
Suppose that condition~\ref{item:B} holds. Then the local coordinate chart $\phi$ induces a diffeomorphism $\psi:=\phi^{-1}|_{\phi(U)}: \phi(U)\to U$ which therefore induces an isomorphism of DG manifolds 
\[(E|_{U}[-1],\iota_{s})\cong (\psi^{\ast}E|_{\phi(U)}, \iota_{s\circ \psi}).\] 
Hence, the Atiyah class of $(E|_{U}[-1],\iota_{s})$ vanishes if and only if the Atiyah class of $(\psi^{\ast}E|_{\phi(U)}, \iota_{s\circ \psi})$ vanishes. 

Since the section $s\circ \psi$ is linear with respect to the given local frame, Eq.~\eqref{eq:AtLocal} implies that the Atiyah class vanishes. 
This concludes the proof of \ref{item:B} $\Rightarrow$ \ref{item:A}.
\end{proof}

It remains to prove \ref{item:A} $\Rightarrow$ \ref{item:B}. 
Since both conditions~\ref{item:A} and~\ref{item:B} are local, we may assume without loss of generality that $U=M=\RR^{n}$ and that condition~\ref{item:A} holds globally.

For the remainder of this section, we adopt the notations from Section~\ref{subsec:local}. Moreover, for simplicity of notation, we use the following notation:
\[\partial_{i}s^{k}=\frac{\partial s^{k}}{\partial x^{i}},\qquad \partial_{ij}s^{k}=\frac{\partial^{2}s^{k}}{\partial x^{i}\partial x^{j}}.\]

We fix $p=0\in \RR^{n}=M$ and assume that $0=p\in s^{-1}(0)$. We denote $r=\rank Ds_{0}$ the rank of $Ds_{0}$ (see Eq.~\eqref{eq:Dsp} for the definition of $Ds_{p}$). By~\cite[Theorem~2.9~(1)]{MR532830} (cf. proof of Lemma~\ref{lem:Calculus}), we may assume that $s^{1}=x^{1}, \ldots, s^{r}=x^{r}$. Altogether, we have
\begin{equation}\label{eq:sr}
  s= x^{1}\epsilon_{1} + \cdots + x^{r} \epsilon_{r} + s^{r+1} \epsilon_{r+1} + \cdots + s^{m} \epsilon_{m}, \quad s^{j}(0)=0,
 \quad\partial_{i}s^{j}(0)=0,\quad \forall i,j > r .
\end{equation}

\subsection{Illustration of the proof}

While the general structure of the proof remains the same, considering the low-dimensional case simplifies indices and reduces potential confusion.
Accordingly, we first prove implication \ref{item:A} $\Rightarrow$ \ref{item:B} under the assumption $\dim M = \rank E = 2$ and $\rank Ds_{0}=1$, making easier for readers to follow the general argument.

\begin{proof}[Illustration of the proof of \ref{item:A} $\Rightarrow$ \ref{item:B}]
By Lemma~\ref{lem:Calculus}, it suffices to prove that $s(0,x_{2})\equiv 0$. Since $s=x^{1}\epsilon_{1}+s^{2}\epsilon_{2}$, it follows that $s(0,x_{2})\equiv 0$ if and only if $s^{2}(0,x_{2})\equiv 0$.

Before we assume the vanishing of the Atiyah class, we begin with some observations.

By Lemma~\ref{lem:AtLocal}, the Atiyah class is represented by 
\[\At=\sum_{i\leq j}\sum_{k}\partial_{ij} s^{k}\cdot dx^{i}\odot dx^{j}\tensor \epsilon_{k}.\]

By Remark~\ref{rem:Free}, there is an ordered basis of the $\smooth{M}$-module $\sections{S^{2}\cotangent{M} \tensor E[-1]}$:
\[(dx^{1}\odot dx^{1}\tensor \epsilon_{1}, dx^{1}\odot dx^{2}\tensor \epsilon_{1}, dx^{2}\odot dx^{2}\tensor \epsilon_{1}, dx^{1}\odot dx^{1}\tensor \epsilon_{2}, dx^{1}\odot dx^{2}\tensor \epsilon_{2}, dx^{2}\odot dx^{2}\tensor \epsilon_{2}),\]
and similarly for each component in the RHS of the identification~\eqref{eq:1-2-0}.

With respect to this basis, the Atiyah cocycle $\At$ can be written as a vector 
\[\At=
\begin{bmatrix}
\partial_{11}s^{1}\\
\partial_{12}s^{1}\\
\partial_{22}s^{1}\\
\partial_{11}s^{2}\\
\partial_{12}s^{2}\\
\partial_{22}s^{2}
\end{bmatrix}
\]
Moreover, with respect to the same basis, the maps $d_{1}, d_{2}, d_{3}$ in Lemma~\ref{lem:AtLocal} can also be written as matrices, such as:
\[
d_{1}=
\begin{bmatrix}
s^{1}&0&0&0&0&0&s^{2}&0&0&0&0&0\\
0&s^{1}&0&0&0&0&0&s^{2}&0&0&0&0\\
0&0&s^{1}&0&0&0&0&0&s^{2}&0&0&0\\
0&0&0&s^{1}&0&0&0&0&0&s^{2}&0&0\\
0&0&0&0&s^{1}&0&0&0&0&0&s^{2}&0\\
0&0&0&0&0&s^{1}&0&0&0&0&0&s^{2}
\end{bmatrix}
\]
\[
d_{2}=
\begin{bmatrix}
\partial_{1}s^{1}  & 0 & 0 & 0 &\partial_{1}s^{2}  & 0 & 0 & 0 \\
\partial_{2}s^{1} & \partial_{1}s^{1}  & 0 & 0 &\partial_{2}s^{2} & \partial_{1}s^{2}  & 0 & 0  \\
0 & \partial_{2}s^{1} & 0 & 0 &0 & \partial_{2}s^{2} & 0 & 0  \\
0 & 0 & \partial_{1}s^{1}  & 0 & 0 & 0 & \partial_{1}s^{2}  & 0 \\
0 & 0 & \partial_{2}s^{1} & \partial_{1}s^{1} & 0 & 0 & \partial_{2}s^{2} & \partial_{1}s^{2}   \\
0 & 0 & 0 & \partial_{2}s^{1}&0 & 0 & 0 & \partial_{2}s^{2} 
\end{bmatrix}
\]
and
\[
d_{3}=
\begin{bmatrix}
\partial_{1}s^{1} & \partial_{2}s^{1} & 0 & 0 & 0 & 0\\
0 & 0 & \partial_{1}s^{1} & \partial_{2}s^{1} & 0 & 0\\
0 & 0 & 0 & 0 & \partial_{1}s^{1} & \partial_{2}s^{1}\\
\partial_{1}s^{2} & \partial_{2}s^{2} & 0 & 0 & 0 & 0 \\
0 & 0 & \partial_{1}s^{2} & \partial_{2}s^{2} & 0 & 0\\
0 & 0 & 0 & 0 & \partial_{1}s^{2} & \partial_{2}s^{2}
\end{bmatrix} .
\]
By Eq.~\eqref{eq:sr} with $r = 1$, the matrices $\At$ and $d_{2}$ take the simplified form shown below:
\begin{equation}\label{eq:BabyAtd2}
\At=
\begin{bmatrix}
0\\
0\\
0\\
\partial_{11}s^{2}\\
\partial_{12}s^{2}\\
\partial_{22}s^{2}
\end{bmatrix},
\qquad
d_{2}=
\begin{bmatrix}
1  & 0 & 0 & 0 &\partial_{1}s^{2}  & 0 & 0 & 0 \\
0 & 1  & 0 & 0 &\partial_{2}s^{2} & \partial_{1}s^{2}  & 0 & 0  \\
0 & 0 & 0 & 0 &0 & \partial_{2}s^{2} & 0 & 0  \\
0 & 0 & 1  & 0 & 0 & 0 & \partial_{1}s^{2}  & 0 \\
0 & 0 & 0 & 1 & 0 & 0 & \partial_{2}s^{2} & \partial_{1}s^{2}   \\
0 & 0 & 0 & 0 &0 & 0 & 0 & \partial_{2}s^{2} 
\end{bmatrix} .
\end{equation}
Note that the image of $d_{1}+ d_{2}+ d_{3}$ contains the standard basis vectors for all rows except the 3rd and 6th.  
In particular, there are no columns in $d_{1}, d_{2}, d_{3}$ with a single 1 either in the 3rd or 6th row and zeros elsewhere, so these rows cannot be cleared independently.
Instead, the entries in these rows appear only in conjunction with entries in other rows via nontrivial linear combinations.

Hence, if the Atiyah class vanishes, then components of the vector $\At$ corresponding to these rows must lie in the image (i.e., $\smooth{M}$-linear span) of these combined columns.  
Concretely, the vanishing of the Atiyah class implies that
\begin{equation}\label{eq:BabyReduced}
\begin{bmatrix}
0 \\
\partial_{22} s^{2}
\end{bmatrix}
\in \image 
\begin{bmatrix}
x^{1} & 0 & s^{2} & 0 & \partial_2 s^{2} & 0 & 1 & 0 \\
0 & x^{1} & 0 & s^{2} & 0 & \partial_2 s^{2} & \partial_1 s^{2} & \partial_2 s^{2}
\end{bmatrix}
\end{equation}
where the first 4 columns come from $d_{1}$, the next 2 from $d_{2}$, and the last 2 from $d_{3}$.

Now, assume that the Atiyah class vanishes. Then, by the observation above, there exists a set of functions $\{f_{1},\ldots, f_{8}\}$ on $M=\RR^{2}$  such that 
\begin{equation}\label{eq:2Rows}
\begin{bmatrix}
0\\
\partial_{22}s^{2}
\end{bmatrix}
=f_{1}
\begin{bmatrix}
x^{1}\\0
\end{bmatrix}
+f_{2}
\begin{bmatrix}
0\\x^{1}
\end{bmatrix}
+f_{3}
\begin{bmatrix}
s^{2}\\0
\end{bmatrix}
+f_{4}
\begin{bmatrix}
0\\s^{2}
\end{bmatrix}
+f_{5}
\begin{bmatrix}
\partial_{2}s^{2}\\0
\end{bmatrix}
+f_{6}
\begin{bmatrix}
0\\\partial_{2}s^{2}
\end{bmatrix}
+f_{7}
\begin{bmatrix}
1\\\partial_{1}s^{2}
\end{bmatrix}
+f_{8}
\begin{bmatrix}
0\\\partial_{2}s^{2}
\end{bmatrix} 
\end{equation}
holds.

By the first row of Eq.~\eqref{eq:2Rows}, we have
\[f_{7}=-f_{1}\cdot x^{1}-f_{3}\cdot s^{2}-f_{5} \cdot \partial_{2}s^{2},\] 
or equivalently, the function $f_{7}$ is a $\smooth{M}$-linear combination of the functions $x^{1}, s^{2}, \partial_{2}s^{2}$. 
Together with the second row of Eq.~\eqref{eq:2Rows}, the function $\partial_{22}s^{2}$ is a $\smooth{M}$-linear combination of the functions $x^{1}, s^{2}, \partial_{2}s^{2}$. That is, 
\begin{equation}\label{eq:222}
\partial_{22}s^{2}\in \langle x^{1},s^{2},\partial_{2}s^{2} \rangle ,
\end{equation}
where the symbol $\langle a,b,c \rangle $ denotes the ideal generated by functions $a,b,c\in \smooth{M}$.

When $x_{1}=0$, Eq.~\eqref{eq:222} implies that there exist functions $f, g\in \smooth{\RR}$ such that 
\begin{equation}\label{eq:ODE222}
\partial_{22}s^{2}(0,x_{2})=f(x_{2})\cdot s^{2}(0,x_{2}) + g(x_{2})\cdot \partial_{2}s^{2}(0,x_{2}).
\end{equation}

Consider an initial value problem in ODE in one variable:
\begin{equation*}
  \begin{bmatrix}
  y_{1}\\ y_{2}
  \end{bmatrix}^{'}
  =
  \begin{bmatrix}
   y_{2}\\f\cdot y_{1}+g \cdot y_{2}
  \end{bmatrix}
\end{equation*}
with the initial value $(y_{1}(0),y_{2}(0))=(0,0)$.

Clearly, $(y_{1}(t),y_{2}(t))\equiv 0$ is a solution. 
On the other hand, by Eqs.~\eqref{eq:ODE222} and~\eqref{eq:sr} (for $n=m=2$ and $r=1$), there is another solution, namely $(y_{1}(t),y_{2}(t))=(s^{2}(0,t),\partial_{2}s^{2}(0,t))$. 
By the uniqueness of solutions to the ODE (Picard--Lindelöf Theorem), we conclude that $s^{2}(0,t)\equiv 0$, and this completes the proof.
\end{proof}

\subsection{Proof of \ref{item:A} $\Rightarrow$ \ref{item:B}}
We prove \ref{item:A} $\Rightarrow$ \ref{item:B} of Lemma~\ref{lem:SubMain} in the general setting: $M=\RR^{n}$, $E=\RR^{n}\times \RR^{m}$ and $\rank Ds_{0}=r$---that is, Eq.~\eqref{eq:sr} holds.

In this general setting, we state a result analogous to the relation~\eqref{eq:222}. 

\begin{claim}\label{claim1}
Assume that Eq.~\eqref{eq:sr} and $\At\in \image(d_{1}+d_{2}+d_{3})$ holds. Then, for each triple $(i,j,k)$ with $i,j,k >r$, we have
\[
\partial_{ij}s^{k} \in \langle x^{1},\ldots, x^{r}, s^{r+1},\ldots, s^{m}, \partial_{i}s^{r+1},\ldots, \partial_{i}s^{m}, \partial_{j}s^{r+1},\ldots, \partial_{j}s^{m}, \partial_{r+1}s^{k},\ldots, \partial_{n}s^{k} \rangle .
\]
\end{claim}

To prove Claim~\ref{claim1}, we first analyse the image of $\liederivative{Q}=d_{1}+d_{2}+d_{3}$ as characterised in Lemma~\ref{lem:AtLocal}.

As in the low-dimensional case, elements of the form $dx^{i}\odot dx^{j}\tensor \epsilon_{k}$, with $1\leq i\leq j \leq n$ and $1\leq k \leq m$ form a basis of the $\smooth{M}$-module $\sections{S^{2}(\cotangent{M})\tensor E[-1]}$. Similarly, we have natural bases for the RHS of identification~\eqref{eq:1-2-0} .

The next two lemmas correspond to $d_{2}$ in Eq.~\eqref{eq:BabyAtd2} from the low-dimensional case.

\begin{lemma}\label{lem:d21}
If either $i\leq r$ or $j \leq r$ holds, then 
\[dx^{i}\odot dx^{j}\tensor \epsilon_{k} \in \image(d_{2})\subset \sections{S^{2}(\cotangent{M})\tensor E[-1]}.\]
\end{lemma}
\begin{proof}
  By Eq.~\eqref{eq:sr}, if $i\leq r$, then $s^{i}=x^{i}$. Then, by Eq.~\eqref{eq:d2}, we have 
  \[d_{2}(d\xi^{i}\tensor dx^{j}\tensor \epsilon_{k})=dx^{i}\odot dx^{j}\tensor \epsilon_{k}.\] 
Similarly, if $j\leq r$, the same result holds by exchanging the roles of $i$ and $j$.
\end{proof}

Together with Lemma~\ref{cor:H1}, this implies that, to determine whether the Atiyah class vanishes, it suffices to consider the coefficients of the elements $dx^{i}\odot dx^{j}\tensor \epsilon_{k}$ with $r<i\leq j \leq n$.

We continue our examination of the image of $d_{2}$. 

Let
\[\pr^{ij}_{k}: \sections{S^{2}(\cotangent{M})\tensor E[-1]} \to \smooth{M}\cdot dx^{i}\odot dx^{j}\tensor \epsilon_{k} \cong \smooth{M}\]
denote the natural projection onto the coefficient of $dx^{i}\odot dx^{j}\tensor \epsilon_{k}$.

\begin{lemma}\label{lem:d22}
For each $r<i \leq j \leq n$,
the image of the composition map $\pr^{ij}_{k}\circ d_{2}$ is the ideal of $\smooth{M}$ generated by $\partial_{i}s^{a}, \partial_{j}s^{a}$ for $a>r$. That is,
\[\image(\pr^{ij}_{k}\circ d_{2}) = \langle \partial_{i} s^{r+1}, \ldots, \partial_{i}s^{m}, \partial_{j}s^{r+1}, \ldots, \partial_{j}s^{m}\rangle \subset \smooth{M} . \]
\end{lemma}
\begin{proof}
Note that, by Eq.~\eqref{eq:d2}, for each triple $(i,j,k)$ with $i\leq j$, we have
\[
\pr^{ij}_{k}\circ d_{2}(d\xi^{a}\tensor dx^{b}\tensor \epsilon_{c})=
\begin{cases}
\partial_{i}s^{a}  & \text{if } (b,c)=(j,k)\\
\partial_{j}s^{a}  & \text{if } (b,c)=(i,k)\\
0 & \text{otherwise.}
\end{cases}
\]
This implies that $\image(\pr^{ij}_{k}\circ d_{2})$ is generated by $\partial_{i}s^{a}$ and $\partial_{j}s^{a}$ for all $a$. However, by Eq.~\eqref{eq:sr}, we have $s^{a}=x^{a}$ for $a\leq r$. In this case, the condition $i,j>r$ implies that 
\[\partial_{i}s^{a}=\partial_{j}s^{a}=0,\]
thus this completes the proof.
\end{proof}

The following lemma is an easy consequence of Lemma~\ref{lem:AtLocal} and Eq.~\eqref{eq:sr}. 
\begin{lemma}\label{lem:d13}
For each $r<i\leq j \leq n$, the following hold:
\begin{enumerate}
\item The Atiyah cocycle $\At$ satisfies
\[\pr^{ij}_{k}(\At)=\begin{cases}
0 &\text{if } k\leq r ,\\
\partial_{ij}s^{k} & \text{if } k>r .
\end{cases}
\]
\item For any $k$, we have
\[ \image(\pr^{ij}_{k}\circ d_{1}) =\langle x^{1},\ldots,x^{r},s^{r+1},\ldots,s^{m} \rangle .\]
\item For any triple $(a,b,c)$ with $a<b$, the composition $\pr^{ij}_{k}\circ d_{3}$ satisfies
\[ \pr^{ij}_{k}\circ d_{3}(dx^{a}\odot dx^{b}\tensor \partial_{c})=\begin{cases}
  1 & \text{if } (a,b,c)=(i,j,k) \text{ and } k\leq r ,\\
\partial_{c}s^{k}  & \text{if } (a,b)=(i,j) \text { and } k>r ,\\
0 & \text{otherwise}.
\end{cases}
\]
\end{enumerate}
\end{lemma}
\begin{proof}
By Eq.~\eqref{eq:sr}, we have $s^{k}=x^{k}$ for $k\leq r$. The first statement holds by Eq.~\eqref{eq:AtLocal}, and the second and third statements follow from direct computation via Eqs.~\eqref{eq:d1} and~\eqref{eq:d3}, respectively.
\end{proof}

\begin{remark}
The lemma above corresponds to the relation~\eqref{eq:BabyReduced} in the low-dimensional case: the first statement corresponds to the LHS, the second statement corresponds to the first 4 columns of the matrix on the RHS, and the third statement to the last 2 columns.
\end{remark}

Let us now prove Claim~\ref{claim1}.
\begin{proof}[Proof of Claim~\ref{claim1}]
  The $\smooth{M}$-module $\sections{\Hom(S^{2}(\tangent{M}), \tangent{M})}$ has a basis consisting of elements of the form
  \[ dx^{i}\odot dx^{j}\tensor \partial_{k}.\]
  Since $\At \in \image(d_{1}+d_{2}+d_{3})$, there exists 
  \[B \in \sections{\Hom(S^{2}(\tangent{M}), \tangent{M})}\]
  such that, for all $1\leq i\leq j \leq n$ and $1\leq k \leq m$,  
\begin{equation}\label{eq:d3B}
\pr_{ij}^{k}(\At-d_{3}(B)) \in \image(\pr_{ij}^{k}\circ (d_{1} + d_{2})).
\end{equation}
By Lemma~\ref{lem:d21}, it suffices to consider the case where $r<i\leq j \leq n$. 
Such a section $B$ can be written in the form
\[B=\sum_{a\leq b}\sum_{c} B_{ab}^{c} \cdot dx^{a}\odot dx^{b}\tensor \partial_{c}\]
for $B_{ab}^{c}\in \smooth{M}$.
By Lemma~\ref{lem:d22} and~\ref{lem:d13}, together with relation~\eqref{eq:d3B}, we obtain the following.
\begin{enumerate}
  \item For $k\leq r$,
\[B_{ij}^{k} \in \langle x^{1},\ldots,x^{r},s^{r+1},\ldots,s^{m}, \partial_{i}s^{r+1},\ldots,\partial_{i}s^{m},\partial_{j}s^{r+1},\ldots,\partial_{j}s^{m} \rangle .\]
  \item For $k>r$, we have 
\[ \partial_{ij}s^{k}-\sum_{c=1}^{n} B_{ij}^{c}\cdot \partial_{c}s^{k} \in \langle x^{1},\ldots,x^{r},s^{r+1},\ldots,s^{m}, \partial_{i}s^{r+1},\ldots,\partial_{i}s^{m},\partial_{j}s^{r+1},\ldots,\partial_{j}s^{m} \rangle .\]
\end{enumerate}
Combining these, for $k>r$, 
\[\partial_{ij}s^{k} \in \sum_{c=r+1}^{n} B_{ij}^{c}\cdot \partial_{c}s^{k} +\langle x^{1},\ldots,x^{r},s^{r+1},\ldots,s^{m}, \partial_{i}s^{r+1},\ldots,\partial_{i}s^{m},\partial_{j}s^{r+1},\ldots,\partial_{j}s^{m} \rangle.\] 
This establishes Claim~\ref{claim1}.
\end{proof}

The following claim follows from the uniqueness of solutions to the ODE.
\begin{claim}\label{claim2}
Assume that Eq.~\eqref{eq:sr} and $\At\in \image(d_{1}+d_{2}+d_{3})$ holds. 
Then $s(x_{1},\ldots, x_{n})=0$ whenever $x_{1}=\ldots=x_{r}=0$.
\end{claim}

\begin{proof}[Proof of Claim~\ref{claim2}]
Let $\widetilde{s}=s|_{\{0\}^{r}\times \RR^{n-r}}$. Then the conclusion of Claim~\ref{claim2} is equivalent to $\widetilde{s}=0$. Since the assumptions on $s$ induce the same condition on $\widetilde{s}$ with $r=0$, it suffices to prove the claim in the case $r=0$.

For any vector $v=(v_{1},\ldots, v_{n})\in \RR^{n}$, define a path $y_{v}:\RR\to \RR^{m+nm}$ by 
\[
  y_{v}(t)=(s^{1}(tv),\ldots, s^{m}(tv),\partial_{1}s^{1}(tv),\ldots,\partial_{1}s^{m}(tv),\partial_{2}s^{1}(tv),\ldots,\partial_{n}s^{m}(tv)).
\]
By Eq.~\eqref{eq:sr}, we have 
\begin{equation}\label{eq:IV}
  y_{v}(0)=0.
\end{equation}
Moreover, by Claim~\ref{claim1}, there exists a smooth map 
\[F_{v}:\RR^{1+m+nm} \to \RR^{m+nm}\]
such that
\begin{equation}\label{eq:ODE}
  \frac{d}{dt}y_{v}(t)=F_{v}(t,y_{v}(t)).
\end{equation}
Indeed, $F_{v}$ can be explicitly written as follows. If we write
\[ \partial_{ij}s^{k}=\sum_{a}f_{ij,a}^{k} \cdot s^{a} + \sum_{b} (g_{ij,b}^{k} \cdot \partial_{i}s^{b}+\hat{g}_{ij,b}^{k}\cdot \partial_{j}s^{b}) + \sum_{c} h_{ij}^{k,c} \cdot \partial_{c}s^{k}\]
for some $f_{ij,a}^{k}, g_{ij,b}^{k}, \hat{g}_{ij,b}^{k}, h_{ij}^{k,c}\in \smooth{\RR^{n}}$, then $F_{v}$ is given by
\[F_{v}(t, p_{1},\ldots, p_{m},q_{11},\ldots, q_{1m},q_{21},\ldots, q_{nm})=(z_{1},\ldots, z_{m},w_{11},\ldots, w_{1m},\ldots, w_{21},\ldots, w_{nm})\]
where
\begin{gather*}
z_{\beta}=\sum_{\alpha} v_{\alpha} \cdot q_{\alpha \beta} ,\\
w_{\beta \gamma} = \sum_{\alpha}v_{\alpha} \Big( \sum_{a}f_{\alpha \beta,a}^{\gamma}\cdot p_{a} + \sum_{b}(g_{\alpha \beta, b}^{\gamma}\cdot q_{\alpha b}+ \hat{g}_{\alpha \beta, b}^{\gamma}\cdot q_{\beta b})+\sum_{c}h_{\alpha \beta}^{\gamma, c} \cdot q_{c \gamma}\Big) .
\end{gather*}

By Eqs.~\eqref{eq:IV} and~\eqref{eq:ODE}, $y_{v}(t)$ solves the initial value problem
\[ y'(t)=F_{v}(t, y(t)),\qquad y(0)=0. \]
However, the identically zero function $y(t)\equiv 0$ is also a solution to this system. By the uniqueness of solutions to the ODE, we must have $y_{v}(t)\equiv 0$. Thus, $s(v)=0$. Since the vector $v$ is chosen arbitrarily, we may conclude that $s=0$, completing the proof of Claim~\ref{claim2}.
\end{proof}

Finally, we prove \ref{item:A} $\Rightarrow$ \ref{item:B} of Lemma~\ref{lem:SubMain} in full generality.

\begin{proof}[Proof of \ref{item:A}$\Rightarrow$ \ref{item:B}]
Assume that condition~\ref{item:A} holds. By shrinking the neighbourhood $U$ if necessary, we may choose a local coordinate chart $\phi:U\to \RR^{n}$ such that Eq.~\eqref{eq:sr} is satisfied for $s\circ \phi^{-1}$. For the rest of the proof, we assume that $M=U=\RR^{n}$ and identify $s$ with $s\circ \phi^{-1}$.

By Lemma~\ref{lem:AtLocal}, the Atiyah class vanishes if and only if $\At\in \image(d_{1}+d_{2}+d_{3})$. Then, by Claim~\ref{claim1} and~\ref{claim2}, the section $s:M\to E$, when restricted to $\{0\}^{r}\times \RR^{n-r}$, vanishes. 

Together with Eq.~\eqref{eq:sr}, this implies that $s$ satisfies the hypothesis of Lemma~\ref{lem:Calculus}, completing the proof of \ref{item:A} $\Rightarrow$ \ref{item:B}.
\end{proof}

\section{The case of derived intersections}
Given two embedded submanifolds $X$ and $Y$ of a smooth manifold $W$, there exists a DG manifold of amplitude $+1$, namely $(\cM_{XY},Q)$, which may be viewed as the derived intersection $X\cap^{h}Y$ of $X$ and $Y$---see \cite{MR4735657} for details. We briefly recall its construction below.

To construct a DG manifold $(\cM_{XY},Q)=(E[-1],\iota_{s})$ of amplitude $+1$, it suffices to construct a vector bundle $E\to M$ and a section $s:E\to M$. We begin with the construction of the base manifold $M$.

Let $\pi:\tangent{W}\to W$ denote the bundle projection. Given an affine connection $\nabla$ on $W$, each tangent vector $v \in \tangent{W}$ induces a unique geodesic
\[\gamma_{v}: I_{v}\to W\]
defined on an open interval $I_{v}\subset \RR$ containing $0$, such that $\gamma_{v}(0)=\pi(v)$ and $\dot{\gamma}_{v}(0)=v$. Define 
\[P_{W}:=\{v\in \tangent{W} : 1\in I_{v}\}, \]
that is, $P_{W}$ consists of those tangent vectors whose associated geodesics are defined at time $t=1$. 
These geodesics are referred to as short geodesics. 
Hence $P_{W}$ is an open submanifold of $\tangent{W}$. We define the base manifold $M$ by
\[M := \{ v\in P_{W} : \gamma_{v}(0)\in X, \quad \gamma_{v}(1)\in Y\}.\]
The manifold $M$ can be viewed as the space of short geodesics in $W$, namely, those $\gamma:[0,1]\to W$ such that $\gamma(0)\in X$ and $\gamma(1)\in Y$.

Next, we construct a vector bundle $E$ over $M$. By abuse of notation, we denote the composition $M\into \tangent{W} \xto{\pi} W$ by the same symbol $\pi$.
Then 
\[E:=\pi^{\ast}\tangent{W}\cong M\times_{W} \tangent{W}\] 
is defined as the pullback of $\tangent{W}$ along $\pi$. 
Geometrically, $E$ can be viewed as a space of covariantly constant vector fields along short geodesics $\gamma \in M$. 

Finally, the section $s:M\to E$ is defined by 
\[s(v)=(v,v)\in M\times_{W} \tangent{W} \cong E\]
for any $v\in M \subset \tangent{W}$. In terms of short geodesics, the section $s$ is characterised by $s(\gamma)=\dot{\gamma}$ for a short geodesic $\gamma$. 

Combining these components, we obtain a DG manifold $(\cM_{XY},Q)=(E[-1],\iota_{s})$ arising from a pair of embedded submanifolds $X$ and $Y$ of $W$. We note that, by \cite[Theorem~3.4]{MR4735657}, this DG manifold is canonically defined, up to isomorphism---that is, it is independent of the choice of affine connections.

The following lemma is also due to \cite{MR4735657}. We include the proof for completeness.

\begin{lemma}[\cite{MR4735657}]\label{lem:XY}
The section $s$ constructed above satisfies $s^{-1}(0)\cong X\cap Y$.
\end{lemma}
\begin{proof}
By construction, the fibre of $E$ at $v\in M$ is $\tangentp{\pi(v)}{W}$. Thus, we have $s(v)=0$ if and only if $v=0\in \tangentp{\pi(v)}{W}$. Moreover, observe that $v=0\in \tangentp{\pi(v)}{W}$ if and only if the corresponding short geodesic is constant. In other words, $\pi(v)=\gamma_{v}(0)=\gamma_{v}(1) \in X\cap Y$. Therefore $s(v)=0$ implies $\pi(v)\in X\cap Y$ and vice versa.
\end{proof}

It is natural to ask about the geometric meaning of the Atiyah class associated with the derived intersection. The following theorem provides an answer to this question.

\begin{theorem}\label{thm:DerivedIntersection}
Given two embedded submanifolds $X$ and $Y$ of a manifold $W$, the Atiyah class of the DG manifold $(\cM_{XY},Q)$ constructed above vanishes if and only if $X$ and $Y$ intersect cleanly.
\end{theorem}

To prove this theorem, we need the following lemma. Recall that for $v\in s^{-1}(0)$, the map $Ds_{v}$ is defined in Eq.~\eqref{eq:Dsp}.

\begin{lemma}\label{lem:dim}
Let $E\to M$ and $s\in \sections{E}$ be as above.
  For $v\in s^{-1}(0)$, denote $p=\pi(v) \in X\cap Y$. Then,
  \[\dim(\tangentp{p}{X}\cap \tangentp{p}{Y}) \leq \dim \ker Ds_{v}.\]
\end{lemma}

\begin{proof}
  Let $\dim W=d$, $\dim X=k$ and $\dim Y=l$.

Let $\phi:U\to \RR^{d}$ be a local coordinate chart near $p$ such that $\phi(p)=0$. 
We may choose $U\subset W$ in such a way that $\phi(U)$ is convex, by shrinking $U$ if necessary. Throughout this proof, we identify $\tangent{(\phi(U))}\cong \phi(U)\times \RR^{d}$.
Also, we write $U_{X}:=X\cap U$ and $U_{Y}:=Y\cap U$. 

Note that, using a partition of unity, any affine connection on $U$ extends to an affine connection on $W$, by shrinking $U$ if necessary. 

Let $\nabla$ be the trivial affine connection on $\phi(U)\subset\RR^{d}$. Define a submanifold $M_{U}$ of $\phi(U)\times \RR^{d}$ by
\[M_{U}=\{(\hat{x},\hat{y}-\hat{x})\in \phi(U)\times \RR^{d} :  \forall \hat{x}\in \phi(U_{X}), \hat{y}\in \phi(U_{Y})\}.\]
Observe that, since $\phi(U)$ is convex, the path 
\[\hat{x}+ t(\hat{y}-\hat{x}): [0,1] \to \phi(U)\]
is a well-defined geodesic on $\phi(U)$.
Thus, the manifold $M_{U}$ serves as a local model of $M$, defined for the embedded submanifolds $U_{X}$ and $U_{Y}$ of $U$, and it is clear that $M_{U}$ is a submanifold of $M$, via the tangent map $\tangent{\phi}$. 

We define 
\[E_{U}:=M_{U}\times_{\phi(U)}\tangent{(\phi(U))} \cong M_{U}\times \RR^{d},\]
and define $s_{U}:M_{U}\to E_{U}$ by
\begin{equation}\label{eq:sU}
  s_{U}(\hat{x},\hat{y}-\hat{x})=(\hat{x},\hat{y}-\hat{x}; \hat{y}-\hat{x})\in M_{U}\times \RR^{d}.
\end{equation}

We constructed $M_{U}$, $E_{U}$ and $s_{U}$ in such a way that the following diagram
\[
\begin{tikzcd}
  M \arrow{r}{s} & E \\
  M_{U}\arrow[hook]{u} \arrow{r}{s_{U}} & E_{U} \arrow[hook]{u}
\end{tikzcd}
\]
is commutative. Since $p\in X\cap Y$, the element $\widetilde{v}:=(\phi(p),0) \in M_{U}$ corresponds to the original element $v\in M$, and thus we obtain an induced commutative diagram:
\[
\begin{tikzcd}
  \tangentp{v}{M} \arrow{r}{Ds_{v}} & E_{v} \arrow{r}{\sim} & \tangentp{p}{U} \arrow{d}{\tangent{\phi}_{p}}  \\
  \tangentp{\widetilde{v}}{M_{U}}\arrow[hook]{u} \arrow{r}{D(s_{U})_{\widetilde{v}}} & (E_{U})_{\widetilde{v}}  \arrow[swap, sloped]{u}{\sim}  \arrow{r}{\sim} & \tangentp{\phi(p)}{(\phi(U))} .
\end{tikzcd}
\]
Thus, we have 
\begin{equation}\label{ineq:dim1}
\dim \ker D(s_{U})_{\widetilde{v}} \leq \dim \ker Ds_{v}.
\end{equation}

Observe that, by the convexity of $\phi(U)$, we have a diffeomorphism 
\[F:U_{X}\times U_{Y} \to M_{U}, \quad (x,y)\mapsto (\phi(x), \phi(y)-\phi(x)),\]
and it satisfies $F(p,p)=\widetilde{v}$. 

By Eq.~\eqref{eq:sU}, it is straightforward to verify that, 
for $(w_{X},w_{Y})\in \tangentp{p}{\big(U_{X}\times U_{Y}\big)}\cong \tangentp{p}{X}\times \tangentp{p}{Y}$, we have
\[D(s_{U})_{\widetilde{v}} \circ \tangent{F}_{(p,p)}(w_{X}, w_{Y})=(\hat{w}_{Y}-\hat{w}_{X}) \in \RR^{d},\]
where $\hat{w}_{X}=\tangent{\phi}_{p}(w_{X}) \in \tangentp{\phi(p)}{(\phi(U))}\cong \RR^{d}$, and similarly for $\hat{w}_{Y}$.

Finally, if $w\in \tangentp{p}{X}\cap \tangentp{p}{Y}$, then $\tangent{F}_{(p,p)}(w,w)\in \ker D(s_{U})_{\widetilde{v}}$. Since $\tangent{F}_{(p,p)}$ is an isomorphism, we have
\[\dim (\tangentp{p}{X}\cap \tangentp{p}{Y}) \leq \dim \ker D(s_{U})_{\widetilde{v}}.\] Combining it with the inequality~\eqref{ineq:dim1} concludes the proof.
\end{proof}

We now prove Theorem~\ref{thm:DerivedIntersection}.
\begin{proof}[Proof of Theorem~\ref{thm:DerivedIntersection}]
By Theorem~\ref{thm:Main}, the Atiyah class of $(\cM_{XY},Q)=(E[-1],\iota_{s})$ vanishes if and only if the intersection of $s$ with the zero section is clean. 
By Lemma~\ref{lem:Rank}, this is equivalent to
\begin{enumerate}
  \item the zero locus $s^{-1}(0)$ is a submanifold;
  \item for every $v\in s^{-1}(0)$, 
    \begin{equation}\label{eq:dim3}
      \dim \tangentp{v}{\big(s^{-1}(0)\big)} + \rank Ds_{v} = \dim \tangentp{v}{M} .
    \end{equation} 
\end{enumerate}

By Lemma~\ref{lem:XY}, the first condition is equivalent to $X\cap Y$ being a submanifold. 
Next, together with the rank theorem and Lemma~\ref{lem:dim}, the second condition implies that
\[ \dim \tangentp{p}{(X\cap Y)}=\dim \tangentp{v}{\big(s^{-1}(0)\big)} = \dim \ker Ds_{v} \geq \dim (\tangentp{p}{X}\cap \tangentp{p}{Y}).\]
Note that, in general, we have $\tangentp{p}{(X\cap Y)}\subset\tangentp{p}{X}\cap \tangentp{p}{Y}$. Thus, the second condition is equivalent to the condition
\[\tangentp{p}{(X\cap Y)} = \tangentp{p}{X}\cap \tangentp{p}{Y}.\]
This proves that the vanishing of the Atiyah class of the DG manifold $(E[-1],\iota_{s})$, arising from the embedded submanifolds $X$ and $Y$ in a manifold $W$, is equivalent to the clean intersection of $X$ and $Y$.
\end{proof}

\section*{Acknowledgements}
The author would like to thank 
Zhuo Chen, Dongwook Choa, Dongnam Ko, Camille Laurent-Gengoux, Hsuan-Yi Liao, Mathieu Stiénon, Maosong Xiang, and Ping Xu
for their interest in this work and their helpful comments. Also, the author is grateful to
the National Center for Theoretical Sciences, the Institut Henri Poincaré, and Pennsylvania State University for their generous support and hospitality.

The author also thanks the anonymous referees for their numerous suggestions and comments.

\printbibliography
\end{document}